\documentclass[11pt,a4paper]{article}%
\usepackage{cite}
\usepackage[utf8]{inputenc}
\usepackage[T1]{fontenc}
\usepackage[top=1cm, bottom=3cm, left=2cm, right=2cm, headheight=0.5cm,
footskip=60pt]{geometry}
\usepackage{amsmath}
\usepackage{amscd}
\usepackage{amsthm}
\usepackage{amssymb}
\usepackage{graphicx}
\usepackage{latexsym}
\usepackage{color}
\usepackage{amsfonts}
\usepackage{chngcntr}
\usepackage{apptools}%
\providecommand{\U}[1]{\protect\rule{.1in}{.1in}}
\newtheorem{theorem}{Theorem}

\newtheorem{corollary}[theorem]{Corollary}

\newtheorem{definition}{Definition}

\newtheorem{lemma}{Lemma}

\newtheorem{proposition}{Proposition}
\newtheorem{remark}{Remark}

\numberwithin{equation}{section}
\numberwithin{remark}{section}
\numberwithin{proposition}{section}
\numberwithin{definition}{section}
\numberwithin{lemma}{section}
\AtAppendix{
\counterwithin{theorem}{section}
\counterwithin{remark}{section}
\counterwithin{proposition}{section}
\counterwithin{equation}{section}
\counterwithin{definition}{section}
}
\begin{document}

\title{Weak Solutions for the Stationary Anisotropic and Nonlocal Compressible
Navier-Stokes System}
\author{D. Bresch\thanks{Univ. Grenoble Alpes, Univ. Savoie Mont-Blanc, CNRS, LAMA,
Chamb\'{e}ry, France; didier.bresch@univ-smb.fr}, \hskip1cm C. Burtea
\thanks{Universit\'{e} Paris Diderot UFR Math\'{e}matiques Batiment Sophie
Germain, Bureau 727, 8 place Aur\'{e}lie Nemours, 75013 Paris;
cburtea@math.univ-paris-diderot.fr}\thinspace}
\maketitle

\begin{abstract}
In this paper, we prove existence of weak solutions for the stationary
compressible Navier-Stokes equations with an anisotropic and nonlocal viscous
stress tensor in a periodic domain ${\mathbb{T}}^{3}$. This gives an answer to
an open problem important for applications in geophysics or in microfluidics.
One of the key ingredients is the new identity discovered by the authors in 
\cite{BrBu} which was used to study the non-stationary anisotropic
compressible Brinkman system.

\end{abstract}

\section{Introduction}

The stationary Navier-Stokes system for a barotropic compressible viscous
fluid reads
\begin{equation}
\left\{
\begin{array}
[c]{l}%
\operatorname{div}\left(  \rho u\right)  =0,\\
\operatorname{div}\left(  \rho u\otimes u\right)  -\mu\Delta u-\left(
\mu+\lambda\right)  \nabla\operatorname{div}u+\nabla p\left(  \rho\right)
=\rho f+g,
\end{array}
\right.  \label{Navier_Stokes_stationary_classic}%
\end{equation}
where $\mu$ and $\lambda$ are given positive constants representing the shear
respectively the bulk viscosities, $f,g\in\mathbb{R}^{3}$ are given exterior
forces acting on the fluid, $\rho\geq0$ is the density, $p\left(  \rho\right)
=a\rho^{\gamma}$ represents the pressure, where $a>0$ and $\gamma\geq1$ are
given constants while $u\in\mathbb{R}^{3}$ is the velocity field. The total
mass of the fluid is given i.e. the above system should be considered along
with the equation
\begin{equation}
\int_{\Omega}\rho=M>0, \label{mass_of_fluid}%
\end{equation}
where $M$ is given.

It is important to point out that all the known mathematical results regarding
the existence of weak solutions for the stationary Navier-Stokes system
strongly use the isotropic and local structure of the viscous stress tensor
owing to the nice algebraic properties it induces for the so-called effective
flux. Extending these results such as to take in account anisotropic or
nonlocal viscous stress-tensors remained an open problem until now.

As explained in \cite{Li}, one cannot expect that
\eqref{Navier_Stokes_stationary_classic}-\eqref{mass_of_fluid} with periodic
boundary conditions to have a solution for any $f,g\in L^{\infty}$ because of
the compatibility condition%
\begin{equation}%
{\displaystyle\int_{\mathbb{T}^{3}}}
(\rho f+g)=0 \label{forces}%
\end{equation}
which comes from integrating the momentum equation. Thus, if $f$ and $g$ have
positive components this would imply that $\rho=0$ which clearly violates the
total mass condition. One way to bypass this structural defect of the periodic
case is to proceed as in \cite{BrezinaNovotny2008} and consider forces $f$
that posses a certain symmetry which ensures the validity of $\left(
\text{\ref{forces}}\right)  $. Another way to bypass this problem was
suggested by P.L. Lions in \cite{Li} and consists in introducing the term
$B\times(B\times u)$ with $B\in L^{\infty}\left(  \mathbb{T}^{3}\right)  $ a
non-constant function in the momentum equation which can be interpreted as the
effect of a magnetic field on the fluid. We claim that the ideas presented in
the present paper can be adapted to handle both situations but in order to
avoid extra technical difficulties we choose to treat the case where $f=0$. We
propose here to investigate the problem of existence of weak solutions
$(\rho,u)$ for the following system:%
\begin{equation}
\left\{
\begin{array}
[c]{l}%
\operatorname{div}\left(  \rho u\right)  =0,\\
\operatorname{div}\left(  \rho u\otimes u\right)  -\mathcal{A}u+a\nabla
\rho^{\gamma}=g,
\end{array}
\right.  \label{Stationary_Stokes}%
\end{equation}
with
\begin{equation}
\rho\geq0,\qquad\int_{\mathbb{T}^{3}}\rho\left(  x\right)  dx=M>0,\qquad
\int_{\mathbb{T}^{3}}u\left(  x\right)  dx=0, \label{Stationary_Stokes_extra}%
\end{equation}
where the viscous diffusion operator $\mathcal{A}$ is given by%
\begin{equation}
\mathcal{A}\cdot\text{ }\mathcal{=}\underset{\text{classical part}%
}{\underbrace{\mu\Delta\cdot+\left(  \mu+\lambda\right)  \nabla
\operatorname{div}\cdot}}+\underset{\text{anisotropic part}}{\underbrace
{\mu\theta\partial_{33}\cdot}}+\underset{\text{nonlocal part}}{\underbrace
{\eta\ast\Delta\cdot+\xi\ast\nabla\operatorname{div}\cdot}}.
\label{diffusion_operator}%
\end{equation}
We will assume the following hypothesis (H):

\begin{itemize}
\item A given total mass $M>0$ of the fluid.

\item An adiabatic constant $\gamma>3$ and a positive constant $a>0$.

\item A forcing term $g$ such that
\[
g\in(L^{\frac{3\left(  \gamma-1\right)  }{2\gamma-1}}\left(  \mathbb{T}%
^{3}\right)  )^{3}\qquad\hbox{ with } \int_{\mathbb{T}^{3}}g=0.
\]

\item The constant $\mu$, $\lambda$ and $\theta$ such that
\[
\mu, \quad\mu+\lambda>0 \qquad\hbox{ and } \qquad\theta>-1.
\]

\item The functions $\eta$ and $\xi$ satisfying
\[
\min\left\{  1,1+\theta\right\}  \mu-\left\Vert \eta\right\Vert _{L^{1}}%
-\frac{1}{3}\left\Vert \xi\right\Vert _{L^{1}}>0 \qquad\text{ or } \qquad
\hat{\eta}\left(  k\right)  ,\hat{\xi}\left(  k\right)  \in\mathbb{R}%
^{+}\text{ for all }k\in\mathbb{Z}^{3}
\]
and
\[
\nabla\eta, \quad\nabla\xi\in L^{2}\left(  \mathbb{T}^{3}\right)  .
\]

\end{itemize}

The main objective of the paper is to prove existence of a weak solution \`{a}
la Leray for the steady compressible barotropic Navier-Stokes system with
anisotropic and nonlocal diffusion. Anisotropic diffusion is present for
instance in geophysical flows, see \cite{Pedlosky2013}, while nonlocal
diffusion is considered when studying confined fluids or in microfluidics
where fluids flows thought narrow vessels. In order to achieve this goal, one
key ingredient is the identity that we proposed in \cite{BrBu} which allowed
us to give a simple proof for the existence of global weak-solutions for the
anisotropic quasi-stationary Stokes system (compressible Brinkman equations).

\medskip

\noindent More precisely, in this paper, we prove the following existence result

\begin{theorem}
\label{Main1} Let us assume Hypothesis \textrm{(H)} be satisfied. There exists
a constant $c_{0}$ such that if
\begin{equation}
(1+|\theta|)\left\vert \theta\right\vert \mu\frac{2\lambda+\mu}{\left(
\lambda+\mu\right)  ^{2}}\leq c_{0}, \label{smallness}%
\end{equation}
then there exists a pair $\left(  \rho,u\right)  \in L^{3\left(
\gamma-1\right)  }\left(  \mathbb{T}^{3}\right)  \times(W^{1,\frac{3\left(
\gamma-1\right)  }{\gamma}}\left(  \mathbb{T}^{3}\right)  )^{3}$ which is a
weak-solution of the stationary Navier-Stokes system \eqref{Stationary_Stokes}--\eqref{Stationary_Stokes_extra}.
\end{theorem}

\begin{remark}
Motivated by physically relevant phenomena like anisotropy or
thermodynamically unstable pressure state laws, D. Bresch and P.E. Jabin
introduced in \textrm{\cite{BrJa}} a new method for the identification of the
pressure in the study of stability of solutions for the non-stationary
compressible Navier-Stokes system. More precisely, if one considers a sequence
of solutions generated by a sequence of initial data for which the
corresponding sequence of initial densities is compact in $L^{1}$, then one is
able to propagate this information for latter times via a nonlocal compactness
criterion modulated with appropriate nonlinear weights. The idea in
\textrm{\cite{BrJa}}, propagation of compactness, is intimately related to the
non-stationary transport equation and it does not seem to adapt to the
stationary case.
\end{remark}

\begin{remark}
The proof of Theorem \ref{Main1} can be adapted to accommodate more general
diffusion operators than $\left(  \text{\ref{diffusion_operator}}\right)  $.
In particular, our method adapts to viscous stress tensors that include
space-dependent coefficients or different convolution kernels for each
component of $u$. In the opinion of the authors, the particular form of
$\mathcal{A}$ proposed in \eqref{diffusion_operator}, besides being physically
relevant, see for instance \textrm{\cite{Eringen1972}} or \textrm{\cite{Er}},
is also relatively easier to manipulate in computations.
\end{remark}

\bigskip

To the authors's knowledge this is the first existence result of weak
solutions taking in consideration anisotropic and nonlocal diffusion for the
steady Navier-Stokes system. The first steps of the proof of Theorem
\ref{Main1} follow a rather well-known path: we consider an elliptic
regularization for the system $\left(  \text{\ref{Stationary_Stokes}}\right)
$ to which classical theory can be applied and therefore we may construct a
sequence of solutions parametrized by the regularization parameter. Of course,
the more delicate part is to recover uniform estimates with respect to the
regularization parameter and to show that the limiting object is a solution of
the stationary Navier-Stokes system. The first key ingredient in the proof of
stability is the new identity discovered by the authors in \cite{BrBu} in the
context of the compressible Brinkman system. As it turns out, the $L^{2}%
$-integrability of the velocity field obtained via the basic energy estimate
is not enough, better integrability is needed in order to justify rigorously
the aforementioned identity. This is achieved by showing that it is possible
to estimate the pressure $\rho^{\gamma}$ in a better space than $L^{2}$: the
smallness Condition \eqref{smallness} is required at this level. We point out
that a similar condition is imposed in \cite{BrJa} in order to treat the
non-stationary compressible Navier-Stokes system. In particular, one can
consider an arbitrary anisotropic amplitude if the bulk viscosity is large enough.

\medskip

The rest of the paper is organized in the following way: in Section
\ref{existing}, we recall existing results concerning weak solutions for the
steady compressible Navier--Stokes equations and we discuss energy dissipation
properties for the diffusion operator $\left(  \text{\ref{diffusion_operator}%
}\right)  $. In Section \ref{NWS} we prove a nonlinear weak stability result,
see Theorem \ref{Main2} bellow, for the stationary compressible Navier-Stokes
equation with anisotropic and nonlocal diffusion operator. In particular, we
show that it is possible to recover strong convergence of the sequence of the
gradients of the velocities and we show how to combine this fact with the
compactness properties of the anisotropic viscous flux in order to identify
the pressure. This is the main idea in the paper. In Section \ref{CAS}, we
propose an approximate system and prove the existence of solutions to such
system. Such approximate system is based on two layers of regularization: one
ensuring ellipticity while the other one providing positivity of the density.
In the last Section \ref{TE}, we prove our main result, Theorem \ref{Main1}
first by establishing uniform estimates with respect to the two parameter and
secondly using the non-linear stability results established in Section
\ref{NWS}.

\section{Existing results for the steady Compressible Navier-Stokes system and
energy dissipation properties for the diffusion operator $\mathcal{A}$
\label{existing}}

\noindent\textit{Existing results.} The problem of constructing solutions for
the above system has been intensively studied and consequently there is a
rather rich literature. We propose below a quick overview of the most recent
results. First of all, we distinguish two types of solutions: strong
respectively weak-solutions. Roughly speaking, a pair $\left(  \rho,u\right)
$ is a strong solution as soon as it verifies
\eqref{Navier_Stokes_stationary_classic}--\eqref{mass_of_fluid} almost
everywhere on the domain of study, see the works of \cite{Beirao1987},
\cite{Padula}, \cite{ValliZaza}. The existence theory of strong solutions
always comes together with some "smallness condition" pertaining either to the
size of the exterior forces $f,g$ \ acting on the system either to the size of
some physical parameters like, for example, the Mach Number see
\cite{DouJiangJiangYang2015}. However, one can prove that this solution is
unique in some sense.

A pair $\left(  \rho,u\right)  $ is weak-solution for $\left(
\text{\ref{Navier_Stokes_stationary_classic}}\right)  $ if it verifies this
system in the sense of distributions and $\rho$ is just a Lebesgue function.
One of the subtle points of the theory of weak-solutions comes from the
genuine nonlinearity induced by the pressure term $p\left(  \rho\right)
=\rho^{\gamma}$ when $\gamma>1$. In order to make things clear we discuss
briefly the most common strategy of constructing weak solutions, namely
approximating system $\left(  \text{\ref{Navier_Stokes_stationary_classic}%
}\right)  $ with an elliptic system, typically by adding $\varepsilon
\Delta\rho$ term in the mass equation. One expects that classical theory for
elliptic equations to give rise to a sequence of solutions indexed by the
approximation parameter $\varepsilon$. Of course, one should be able to obtain
estimates verified by the sequence $\left(  \rho^{\varepsilon},u^{\varepsilon
}\right)  $ uniformly with respect to $\varepsilon$ and to show that the limit
is a solution to the $\left(  \text{\ref{Navier_Stokes_stationary_classic}%
}\right)  $. However, we cannot reasonably expect to recover any regularity on
$\rho^{\varepsilon}$, one is able only to recover that $\rho^{\varepsilon}$ is
uniformly bounded in a Lebesgue space with integrability index greater than
$\gamma$. Thus, as weak convergence is not commuting with nonlinear functions
a delicate point is to be able to recover that the weak limit of the pressure
sequence is the pressure associated to the limit density. This point proved to
be difficult and the problem of existence of weak solutions resisted until
$1998$ when P.L. Lions in \cite{Li} proposed a solution combing two ingredients:

\begin{itemize}
\item renormalized transport theory which consists in the rigorous
justification of the fact that $\rho$ also verifies
\[
\operatorname{div}\left(  b\left(  \rho\right)  u\right)  +\left(  \rho
b^{\prime}\left(  \rho\right)  -b\left(  \rho\right)  \right)
\operatorname{div}u=0,
\]
for any $b$ sufficiently "well-behaved".

\item the compactness properties of the so-called effective flux
\[
F=\left(  2\mu+\lambda\right)  \operatorname{div}u-p\left(  \rho\right)  .
\]
It is easy to give some rather informal hints why the above quantity behaves
well: applying the divergence operator in the momentum equation of $\left(
\text{\ref{Navier_Stokes_stationary_classic}}\right)  $ we get that%
\[
-\Delta F=\operatorname{div}\left(  \rho f+g\right)  -\operatorname{div}%
\left(  \rho u\cdot\nabla u\right)
\]
and thus $\nabla F$ is of the same order as $\rho f+g-\rho u\cdot\nabla u$. 
\end{itemize}

We note that the effective-flux, was used before in the context of the $1D$
non-stationary Navier-Stokes system by J. Smoller and D. Hoff
\cite{HoffSmoller1985} and by D. Serre \cite{Serre1991} when studying the
problem of propagation of oscillations. In the multi-dimensional situation
P.L. Lions used these two features in order to prove that%
\[
\lim_{\varepsilon}b\left(  \rho^{\varepsilon}\right)  \left(  \left(
2\mu+\lambda\right)  \operatorname{div}u^{\varepsilon}-p\left(  \rho
^{\varepsilon}\right)  \right)  =\lim_{\varepsilon}b\left(  \rho^{\varepsilon
}\right)  \lim_{\varepsilon}\left(  \left(  2\mu+\lambda\right)
\operatorname{div}u^{\varepsilon}-p\left(  \rho^{\varepsilon}\right)  \right)
,
\]
for any $b$ sufficiently "well-behaved" where $\left(  \rho^{\varepsilon
},u^{\varepsilon}\right)  $ is a sequence of solutions for the Navier-Stokes
system. This identity is used in order to compare $\lim_{\varepsilon}%
\psi\left(  \rho^{\varepsilon}\right)  $ with $\psi\left(  \lim_{\varepsilon
}\rho^{\varepsilon}\right)  $ for an appropriate convex function $\psi$ and to
recover compactness for the density sequence. This nice argument, is of great
generality: if $\gamma$ is large enough, it can be used to prove compactness
for the density regardless of the domain where the problem of existence is
studied and even for the non-stationary version of system $\left(
\text{\ref{Navier_Stokes_stationary_classic}}\right)  $. In \cite{Li}, P.L.
Lions constructed weak-solutions for system $\left(
\text{\ref{Navier_Stokes_stationary_classic}}\right)  $ if $\gamma>5/3$ in the
case of finite domains with Dirichlet boundary condition for the velocity, in
the whole space case $\mathbb{R}^{3}$, in the periodic boundary conditions and
the case of an exterior domain. At this point it is worth mentioning that
physical relevant values for the adiabatic coefficients include $\gamma=5/3$
for monatomic gases, $\gamma=7/5$ for diatomic gases, $\gamma=4/3$ for
polyatomic gases. An argument leading to the relaxation of the condition
$\gamma>5/3$ is due to S. Novo and A. Novotn\'{y} \cite{NoNo1} where the
authors obtain existence of weak solutions for $\gamma>3/2$ and Dirichlet
boundary conditions with potential body forces $f=\nabla h\in L^{\infty}$. It
is worth mentioning that their argument relies in a crucial manner on E.
Feireisl's work \cite{Feireisl2001} on the non-stationary version of $\left(
\text{\ref{Navier_Stokes_stationary_classic}}\right)  $ where he introduced
and studied a defect measure constructed with the help of truncations of the density.

The next improvement on the admissible bound on $\gamma$ came in the context
of the periodic boundary conditions. More precisely, J. B\v{r}ezina and A.
Novotn\'{y} \cite{BrezinaNovotny2008} constructed weak-solution for
$\gamma>\left(  1+\sqrt{13}\right)  /3\approx1.53$ for volume non-potential
body forces respectively for $\gamma>\left(  3+\sqrt{41}\right)
/8\approx1.175$ in the case of potential body forces. Finally, the optimal
result in the periodic framework, existence for $\gamma>1$ was obtained in
\cite{JiangZhou2011(a)} by S.~Jiang, and C. Zhou. Concerning finite domains
with Dirichlet boundary condition, the optimal result regarding the value of
$\gamma$ is due to P.~Plotnikov and W.~Weigant \cite{PlotnikovWeigant2015} who
constructed solutions for any $f\in L^{\infty}\left(  \Omega\right)  $, $g=0$
with pressure functions $p\left(  \rho\right)  =\rho^{\gamma}$ for any
$\gamma>1$, improving upon previous preliminary results obtained in
\cite{PlotnikovSokolowski2007} where the total mass condition $\left(
\text{\ref{mass_of_fluid}}\right)  $ was replaced by $\int_{\Omega}\rho\left(
x\right)  d\left(  x\right)  ^{-s}dx=M$ where $d\left(  x\right)  $ is the
distance from $x$ to the boundary of the domain or
\cite{FrehseSteinhauerWeigant2012} where the Dirichlet problem was solved for
$\gamma>4/3$. We also mention results dealing with the relaxation of the
conditions for the regularity of the boundary \cite{NoNo2005} or the case or
non-compact boundaries \cite{NoNoPo}. The problem with the non-penetration
condition $u\cdot n=0$ where $n$ is the unit normal at the boundary along with
slip boundary conditions on the velocity was studied by M.\ Pokorn\'{y} and
P.B. Mucha in \cite{PokornyMucha2008} where they are able to construct
solutions with bounded density $\rho\in L^{\infty}$ in the case $\gamma>3$.
More recently, E. Feireisl and A. Novotn\'{y} \cite{FeireislNovotny2018}
showed the existence of weak solutions for general inflow, outflow boundary
conditions and monotone pressures that become singular near a finite value
$\bar{\rho}$. For a survey on results obtained prior to the year $2003$ one
can consult the book of A. Novotn\'{y} and I. Stra\v{s}kraba \cite{NoSt}. We
emphasize that all the results obtained in the papers previously cited concern
an isotropic linear viscous diffusion:
\[
\mathcal{A}u=\mu\Delta u+(\mu+\lambda)\nabla\mathrm{div}u.
\]
The objective of our paper is to enlarge the choice of the viscous stress
tensors allowing more general diffusion operators of the form
\eqref{diffusion_operator}. Denoting
\begin{equation}
\Delta_{\theta}\overset{not.}{=}\Delta+\theta\partial_{33},
\label{delta_theta}%
\end{equation}
we can easily see that\footnote{We denote by $\eta\ast\Delta u$ respectively
$\xi\ast\nabla\operatorname{div}u$ the $3D$-vector fields with components
$\left(  \eta\ast\Delta u^{1},\eta\ast\Delta u^{2},\eta\ast\Delta
u^{3}\right)  $ respectively $\left(  \xi\ast\partial_{1}\operatorname{div}%
u,\xi\partial_{2}\operatorname{div}u,\xi\partial_{3}\operatorname{div}%
u\right)  $ where $u=\left(  u^{1},u^{2},u^{3}\right)  $.}%
\[
\mathcal{A}u=\mu\Delta_{\theta}u+\left(  \mu+\lambda\right)  \nabla
\operatorname{div}u+\eta\ast\Delta u+\xi\ast\nabla\operatorname{div}u.
\]

\medskip

\noindent\textit{Energy dissipation for the new diffusion with anisotropic
coefficients and nonlocal terms.} Let us now discuss energy dissipation for
the viscous operator that will be crucial to ensure existence of weak
solution. In the following we use the notations%
\begin{equation}
\nabla_{\theta}=\left(  \partial_{1},\partial_{2},(1+\theta)^{\frac{1}{2}%
}\partial_{3}\right)  , \label{grad_theta}%
\end{equation}
and%
\begin{equation}
\operatorname{div}_{\theta}u=\partial_{1}u^{1}+\partial_{2}u^{2}+\left(
1+\theta\right)  \partial_{3}u^{3}. \label{div_theta}%
\end{equation}
We can write\footnote{By $\nabla_{\theta}u$ we denote the $3\times3$ matrix
whose $i^{th}$ column is $\nabla_{\theta}u^{i}$. For two matrices $A,B$, we
denote by $A:B=a_{ij}b_{ij}$.}
\begin{equation}
\left.
\begin{array}
[c]{l}%
\left\langle \mathcal{A}u,u\right\rangle =\frac{1}{2}\mu\Delta_{\theta}\left(
\left\vert u\right\vert ^{2}\right)  -\mu\nabla_{\theta}u:\nabla_{\theta}u\\
\text{ \ \ \ \ \ \ \ \ \ \ \ \ \ \ \ }+\left(  \mu+\lambda\right)
\operatorname{div}\left(  u\operatorname{div}u\right)  -\left(  \mu
+\lambda\right)  (\operatorname{div}u)^{2}\\
\text{ \ \ \ \ \ \ \ \ \ \ \ \ \ \ \ }+\operatorname{div}(\eta\ast\nabla
uu)-\left(  \eta\ast\nabla u\right)  :\nabla u+\operatorname{div}\left(
u\xi\ast\operatorname{div}u\right)  -\xi\ast\operatorname{div}%
u\operatorname{div}u\\
\text{ \ \ \ \ \ \ \ \ \ \ }=\mathcal{B}\left(  u,u\right)  -\mathcal{C}%
\left(  u,u\right)  ,
\end{array}
\right.  \label{Au_scalaire_u}%
\end{equation}
with
\begin{equation}
\left\{
\begin{array}
[c]{l}%
\mathcal{B}\left(  u,u\right)  \overset{def.}{=}\frac{1}{2}\mu\Delta_{\theta
}\left(  \left\vert u\right\vert ^{2}\right)  +\left(  \mu+\lambda\right)
\operatorname{div}\left(  u\operatorname{div}u\right) \\
\text{ \ \ \ \ \ \ \ \ \ \ \ \ \ \ }+\operatorname{div}(\eta\ast\nabla
uu)-\left(  \eta\ast\nabla u\right)  :\nabla u+\operatorname{div}\left(
\left(  \xi\ast\operatorname{div}u\right)  \text{ }u\right)  -\left(  \xi
\ast\operatorname{div}u\right)  \operatorname{div}u,\\
\mathcal{C}\left(  u,u\right)  \overset{def.}{=}\mu\nabla_{\theta}%
u:\nabla_{\theta}u+\left(  \mu+\lambda\right)  (\operatorname{div}u)^{2}.
\end{array}
\right.  \label{definition_B_et_C}%
\end{equation}
Let us observe that if
\[
u^{\varepsilon}\rightarrow u\text{ strongly in }\left(  L^{2}\left(
\mathbb{T}^{3}\right)  \right)  ^{3}\text{ and }\nabla u^{\varepsilon
}\rightharpoonup\nabla u\text{ weakly in }\left(  L^{2}\left(  \mathbb{T}%
^{3}\right)  \right)  ^{9}%
\]
then
\begin{equation}
\mathcal{B}\left(  u^{\varepsilon},u^{\varepsilon}\right)  \rightharpoonup
\mathcal{B}\left(  u,u\right)  \text{ in the sense of distributions,}
\label{convergence_property}%
\end{equation}
a fact that will prove crucial in our analysis. If the first condition on
$\mathcal{A}$ in Hypothesis (H) holds true then, we have that%
\begin{align}
-\int_{\mathbb{T}^{3}}\left\langle \mathcal{A}u,u\right\rangle  &
=\int_{\mathbb{T}^{3}}\left\{  \mu\nabla_{\theta}u:\nabla_{\theta}u+\left(
\mu+\lambda\right)  (\operatorname{div}u)^{2}\right\}  +\int_{\mathbb{T}^{3}%
}\left(  \eta\ast\nabla u\right)  :\nabla u+\int_{\mathbb{T}^{3}}\xi
\ast\operatorname{div}u\operatorname{div}u\nonumber\\
&  \geq\left(  \min\{1,1+\theta\}\mu-\left\Vert \eta\right\Vert _{L^{1}}%
-\frac{1}{3}\left\Vert \xi\right\Vert _{L^{1}}\right)  \int_{\mathbb{T}^{3}%
}\nabla u:\nabla u+\left(  \mu+\lambda\right)  \int_{\mathbb{T}^{3}%
}(\operatorname{div}u)^{2}. \label{ellipticity_of_diffusion_operator_1}%
\end{align}
If the second condition on $\mathcal{A}$ in Hypothesis (H) holds true then, we
have that%
\begin{align}
-\int_{\mathbb{T}^{3}}\left\langle \mathcal{A}u,u\right\rangle  &
=\int_{\mathbb{T}^{3}}\left\{  \mu\nabla_{\theta}u:\nabla_{\theta}u+\left(
\mu+\lambda\right)  (\operatorname{div}u)^{2}\right\}  +\int_{\mathbb{T}^{3}%
}\eta\ast\nabla u:\nabla u+\int_{\mathbb{T}^{3}}\xi\ast\operatorname{div}%
u\operatorname{div}u\nonumber\\
&  \geq\min\{1,1+\theta\}\mu\int_{\mathbb{T}^{3}}\nabla u:\nabla u+\sum
_{k\in\mathbb{Z}^{3}}\sum_{i,j}\hat{\eta}\left(  k\right)  \left\vert
\widehat{\partial_{j}u^{i}}\left(  k\right)  \right\vert ^{2}+\sum
_{k\in\mathbb{Z}^{3}}\hat{\xi}\left(  k\right)  \left\vert \widehat
{\operatorname{div}u}\left(  k\right)  \right\vert ^{2}.
\label{ellipticity_of_diffusion_operator_2}%
\end{align}
Remark finally that
\begin{equation}
\operatorname{div}\mathcal{A}u=\left(  \mu\Delta_{\theta}+\left(  \mu
+\lambda\right)  \Delta\right)  \operatorname{div}u+\Delta\left(  (\eta
+\xi)\ast\operatorname{div}u\right)  , \label{div_Au}%
\end{equation}
and%
\begin{equation}
\operatorname{div}_{\theta}\mathcal{A}u=\Delta_{\theta}(\mu\operatorname{div}%
_{\theta}u+\left(  \mu+\lambda\right)  \operatorname{div}u+\xi\ast
\operatorname{div}u)+\Delta\eta\ast\operatorname{div}_{\theta}u.
\label{div_theta_Au}%
\end{equation}

\section{Nonlinear Weak stability \label{NWS}}

This part of the paper concerns the nonlinear weak stability of the steady
compressible Navier-Stokes system with anisotropic coefficients and nonlocal
terms in the stress tensor. More precisely, we prove

\begin{theorem}
\label{Main2} Let Hypothesis \textbf{(H) } with an external force
$g^{\varepsilon}$ be uniformly satisfied with respect to $\varepsilon$. Let
$\left(  \rho^{\varepsilon},u^{\varepsilon}\right)  _{\varepsilon>0}$ be a
sequence of weak solutions of
\begin{equation}
\left\{
\begin{array}
[c]{l}%
\operatorname{div}\left(  \rho^{\varepsilon}u^{\varepsilon}\right)  =0,\\
\operatorname{div}\left(  \rho^{\varepsilon}u^{\varepsilon}\otimes
u^{\varepsilon}\right)  -\mathcal{A}u^{\varepsilon}+\nabla(\rho^{\varepsilon
})^{\gamma}=g^{\varepsilon},\\
\int_{\mathbb{T}^{3}}\rho^{\varepsilon}\left(  x\right)  dx=M,\text{ }%
\rho^{\varepsilon}\geq0,\\
\int_{\mathbb{T}^{3}}u^{\varepsilon}\left(  x\right)  dx=0,
\end{array}
\right.
\end{equation}
satisfying%
\begin{equation}
\left\Vert \rho^{\varepsilon}\right\Vert _{L^{3\left(  \gamma-1\right)
}(\mathbb{T}^{3})}+\left\Vert \nabla u^{\varepsilon}\right\Vert _{L^{\frac
{3\left(  \gamma-1\right)  }{\gamma}}(\mathbb{T}^{3})}\leq C,
\end{equation}
where $C>0$ is a constant independent of $\varepsilon$. Then, there exists
$\left(  \rho,u\right)  \in L^{3\left(  \gamma-1\right)  }(\mathbb{T}%
^{3})\times(L^{\frac{3\left(  \gamma-1\right)  }{\gamma}})^{3}(\mathbb{T}%
^{3})$ such that up to a subsequence%
\[
\left\{
\begin{array}
[c]{l}%
\rho^{\varepsilon}\rightharpoonup\rho\text{ weakly in }L^{3\left(
\gamma-1\right)  }\left(  \mathbb{T}^{3}\right)  ,\\
\rho^{\varepsilon}\rightarrow\rho\text{ in }L^{r}\left(  \mathbb{T}%
^{3}\right)  \text{ for all }r\in\lbrack1,3(\gamma-1)),\\
u^{\varepsilon}\rightarrow u\text{ in }L^{r}\left(  \mathbb{T}^{3}\right)
\text{ for all }r\in\lbrack1,3(\gamma-1)),\\
\nabla u^{\varepsilon}\rightharpoonup\nabla u\text{ weakly in }(L^{\frac
{3\left(  \gamma-1\right)  }{\gamma}}\left(  \mathbb{T}^{3}\right)  )^{9},\\
\nabla u^{\varepsilon}\rightarrow\nabla u\text{ strongly in }(L^{r}\left(
\mathbb{T}^{3}\right)  )^{9}\text{ for all }r\in\lbrack1,\frac{3\left(
\gamma-1\right)  }{\gamma}),
\end{array}
\right.
\]
with $\left(  \rho,u\right)  $ a weak solution of the stationary compressible
system \eqref{Stationary_Stokes}--\eqref{Stationary_Stokes_extra}.
\end{theorem}

The proof of Theorem \ref{Main2} is rather non-standard in the context of
problems coming from compressible fluid mechanics: we are able to prove that
the sequence of velocity gradients converges strongly and recover a posteriori
compactness properties of the equivalent anisotropic effective-flux. The main
ingredient is the identity
\begin{equation}
\operatorname{div}\left(  \left(  \overline{\rho^{\gamma}}-\rho^{\gamma
}\right)  ^{\frac{1}{\gamma}}\right)  +\left(  \overline{\mathcal{C}\left(
u,u\right)  }-\mathcal{C}\left(  u,u\right)  \right)  \left(  \overline
{\rho^{\gamma}}-\rho^{\gamma}\right)  ^{\frac{1}{\gamma}-1}=0,
\label{relation}%
\end{equation}
where $\overline{\rho^{\gamma}}=\lim_{\varepsilon}\left(  \rho^{\varepsilon
}\right)  ^{\gamma}$, $\overline{\mathcal{C}\left(  u,u\right)  }%
=\lim_{\varepsilon}\mathcal{C}\left(  u^{\varepsilon},u^{\varepsilon}\right)
$ where $\mathcal{C}\left(  u,u\right)  $ is defined in $\left(
\text{\ref{definition_B_et_C}}\right)  $. As usually in PDEs, a stability
property is the first important step before proving the existence of weak
solutions. This is also the case in the present situation where it turns out
that we can adapt the arguments used in Theorem \ref{Main2} in order to obtain
an existence result. This will be the subject of Section \ref{TE}. Let us
mention that the above identity is very sensitive to the specific form of the
pressure as a power function. We are not able to treat the case of more
general monotone pressure laws.

\textit{Proof of Theorem \ref{Main2}.} Consider $\left(  \rho^{\varepsilon
},u^{\varepsilon}\right)  _{\varepsilon>0}$ a sequence verifying
\begin{equation}
\left\{
\begin{array}
[c]{l}%
\operatorname{div}\left(  \rho^{\varepsilon}u^{\varepsilon}\right)  =0,\\
\operatorname{div}\left(  \rho^{\varepsilon}u^{\varepsilon}\otimes
u^{\varepsilon}\right)  -\mathcal{A}u^{\varepsilon}+\nabla(\rho^{\varepsilon
})^{\gamma}=g^{\varepsilon},\\%
{\displaystyle\int_{\mathbb{T}^{3}}}
\rho^{\varepsilon}\left(  x\right)  dx=M,\text{ }%
{\displaystyle\int_{\mathbb{T}^{3}}}
u^{\varepsilon}\left(  x\right)  dx=0,\text{ }\rho^{\varepsilon}\geq0,
\end{array}
\right.  \label{system}%
\end{equation}
along with the following estimates%
\begin{equation}
\left\Vert \rho^{\varepsilon}\right\Vert _{L^{3\left(  \gamma-1\right)
}(\mathbb{T}^{3})}+\left\Vert \nabla u^{\varepsilon}\right\Vert _{L^{\frac
{3\left(  \gamma-1\right)  }{\gamma}}(\mathbb{T}^{3})}\leq C,
\label{uniform_bounds}%
\end{equation}
where $C$ is independent of $\varepsilon$. Classical functional analysis
results allow us to get the existence of functions $\left(  \rho
,u,\overline{\rho^{\gamma}},\overline{\mathcal{C}\left(  u,u\right)  }\right)
$ such that up to a subsequence%
\begin{equation}
\left\{
\begin{array}
[c]{l}%
\rho^{\varepsilon}\rightharpoonup\rho\text{ weakly\ in }L^{3(\gamma-1)}\left(
\mathbb{T}^{3}\right)  ,\\
(\rho^{\varepsilon})^{\gamma}\rightharpoonup\overline{\rho^{\gamma}%
}\text{\ weakly in }L^{\frac{3\left(  \gamma-1\right)  }{\gamma}}\left(
\mathbb{T}^{3}\right)  ,\\
\nabla u^{\varepsilon}\rightharpoonup\nabla u\text{ weakly in }L^{\frac
{3\left(  \gamma-1\right)  }{\gamma}}\left(  \mathbb{T}^{3}\right)  ,\\
\mathcal{C}\left(  u^{\varepsilon},u^{\varepsilon}\right)  \rightharpoonup
\overline{\mathcal{C}\left(  u,u\right)  }\text{ weakly in }L^{\frac{3\left(
\gamma-1\right)  }{2\gamma}}\left(  \mathbb{T}^{3}\right)  ,\\
u^{\varepsilon}\rightarrow u\text{ strongly in }L^{q}\left(  \mathbb{T}%
^{3}\right)  \text{ for any }1\leq q<3\left(  \gamma-1\right)  .
\end{array}
\right.  \label{convergence_properties}%
\end{equation}
We deduce that%
\begin{equation}
\left\{
\begin{array}
[c]{l}%
\operatorname{div}\left(  \rho u\right)  =0,\\
\operatorname{div}\left(  \rho u\otimes u\right)  -\mathcal{A}u+\nabla
\overline{\rho^{\gamma}}=g,\\%
{\displaystyle\int_{\mathbb{T}^{3}}}
\rho\left(  x\right)  dx=M,%
{\displaystyle\int_{\mathbb{T}^{3}}}
u\left(  x\right)  dx=0,\text{ }\rho^{\varepsilon}\geq0.
\end{array}
\right.  \label{limit_before_identification}%
\end{equation}
The more delicate problem is to be able to identify $\overline{\rho^{\gamma}}$
with $\rho^{\gamma}$. Let us explain the main ideas concerning the
identification of the pressure in the isotropic case and then in the
anisotropic case.

\medskip\noindent\textit{Identification of the pressure in the isotropic case}

Let us briefly sketch the idea behind P.L. Lions's proof in the case when
$\theta=\lambda=0$ and $\eta=\xi=0$, when the system reduces to%
\begin{equation}
\left\{
\begin{array}
[c]{l}%
\operatorname{div}\left(  \rho^{\varepsilon}u^{\varepsilon}\right)  =0,\\
\operatorname{div}\left(  \rho^{\varepsilon}u^{\varepsilon}\otimes
u^{\varepsilon}\right)  -\mu\Delta u^{\varepsilon}+\nabla(\rho^{\varepsilon
})^{\gamma}=g^{\varepsilon},\\%
{\displaystyle\int_{\mathbb{T}^{3}}}
\rho^{\varepsilon}\left(  x\right)  dx=M,\text{ }%
{\displaystyle\int_{\mathbb{T}^{3}}}
u^{\varepsilon}\left(  x\right)  dx=0,\text{ }\rho^{\varepsilon}\geq0.
\end{array}
\right.  \label{isotropic}%
\end{equation}
As we allready mentioned in the introduction, there are two important points:
first the regularity of the \textit{effective flux} defined as%
\begin{equation}
F^{\varepsilon}\overset{def.}{=}\mu\operatorname{div}u^{\varepsilon}%
-(\rho^{\varepsilon})^{\gamma}. \label{flux}%
\end{equation}
Indeed, applying the divergence operator in the momentum equation gives us%
\[
-\Delta F^{\varepsilon}=-\operatorname{div}\left(  \rho^{\varepsilon
}u^{\varepsilon}\cdot\nabla u^{\varepsilon}\right)  +\operatorname{div}%
g^{\varepsilon}.
\]
Thus $(\nabla F^{\varepsilon})_{\varepsilon>0}$ is uniformly bounded in
$W^{1,\frac{3\left(  \gamma-1\right)  }{2\gamma-1}}\left(  \mathbb{T}%
^{3}\right)  $ and owing to the Rellich-Kondrachov theorem we obtain that%
\[
b\left(  \rho^{\varepsilon}\right)  \cdot F^{\varepsilon}\rightharpoonup
\overline{b\left(  \rho\right)  }\cdot F\text{ weakly in }L^{1}\left(
\mathbb{T}^{3}\right)  \text{,}%
\]
for any continuous $b$ verifying some growth properties in $0$ and at infinity
where $b\left(  \rho^{\varepsilon}\right)  \rightharpoonup\overline{b\left(
\rho\right)  }$. The second part of the proof makes a clever use of the above
identify. More precisely, fix a $\theta\in]0,1[$. Owing to Proposition
\ref{Prop_ren2} we get that%
\[
\operatorname{div}\left(  (\rho^{\varepsilon})^{\theta}u^{\varepsilon}\right)
+\left(  \theta-1\right)  (\rho^{\varepsilon})^{\theta}\operatorname{div}%
u^{\varepsilon}=0,
\]
which rewrites as%
\begin{align*}
&  \mu\operatorname{div}\left(  (\rho^{\varepsilon})^{\theta}u^{\varepsilon
}\right)  +\left(  \theta-1\right)  (\rho^{\varepsilon})^{\theta}%
(\mu\operatorname{div}u^{\varepsilon}-(\rho^{\varepsilon})^{\gamma})+\left(
\theta-1\right)  (\rho^{\varepsilon})^{\theta+\gamma}\\
&  =\mu\operatorname{div}\left(  (\rho^{\varepsilon})^{\theta}u^{\varepsilon
}\right)  +\left(  \theta-1\right)  (\rho^{\varepsilon})^{\theta
}F^{\varepsilon}+\left(  \theta-1\right)  (\rho^{\varepsilon})^{\theta+\gamma
}=0,
\end{align*}
such that passing to the limit yields%
\[
\mu\operatorname{div}\left(  \overline{\rho^{\theta}}u\right)  +\left(
\theta-1\right)  \overline{\rho^{\theta}}(\mu\operatorname{div}u-\overline
{\rho^{\gamma}})+\left(  \theta-1\right)  \overline{\rho^{\theta+\gamma}}=0.
\]
Using once more Proposition \ref{Prop_ren2} we get that%
\[
\mu\operatorname{div}\left(  \overline{\rho^{\theta}}^{\frac{1}{\theta}%
}u\right)  =\left(  \frac{1}{\theta}-1\right)  \left(  \overline{\rho
^{\theta+\gamma}}-\overline{\rho^{\theta}}\overline{\rho^{\gamma}}\right)
\overline{\rho^{\theta}}^{\frac{1}{\theta}-1}.
\]
But by integration we get that%
\begin{equation}
\int_{\mathbb{T}^{3}}\left(  \overline{\rho^{\theta+\gamma}}-\overline
{\rho^{\theta}}\overline{\rho^{\gamma}}\right)  \overline{\rho^{\theta}%
}^{\frac{1}{\theta}-1}=0, \label{compacteness_iso}%
\end{equation}
which, by the positivity of the integrand implies that%
\[
\overline{\rho^{\theta+\gamma}}-\overline{\rho^{\theta}}\overline{\rho
^{\gamma}}=0,
\]
which implies by monotone operator theory that $\rho=\overline{\rho^{\gamma}%
}^{\frac{1}{\gamma}}$.

\medskip\noindent\textit{Identification of the pressure in the anisotropic
case.}

The change of the algebraic structure of the effective flux in the anisotropic
will make it impossible to adapt in a trivial manner the above approach. In
order to highlight the differences with the isotropic case, in the following
lines we continue our discussion for the case when $\theta>-1$, $\theta
\not =0$, $\lambda=0$ and $\eta=\xi=0$. There are two ways one can think of
the anisotropic-effective flux. First, as explained in \cite{BrJa}, we just
take the divergence of the momentum equation and to write it as%
\[
-\Delta_{\theta}\left(  \mu\operatorname{div}u^{\varepsilon}-\left(
\int_{\mathbb{T}^{3}}(\rho^{\varepsilon})^{\gamma}+\Delta_{\theta}^{-1}%
\Delta\left(  (\rho^{\varepsilon})^{\gamma}-\int_{\mathbb{T}^{3}}%
(\rho^{\varepsilon})^{\gamma}\right)  \right)  \right)  =\operatorname{div}%
g^{\varepsilon}+\operatorname{div}\left(  \rho^{\varepsilon}u^{\varepsilon
}\cdot\nabla u^{\varepsilon}\right)
\]
and to try to mimic the proof in the isotropic case using%
\begin{equation}
F_{an}^{\varepsilon}=\mu\operatorname{div}u^{\varepsilon}-\left(
\int_{\mathbb{T}^{3}}(\rho^{\varepsilon})^{\gamma}+\Delta_{\theta}^{-1}%
\Delta\left(  (\rho^{\varepsilon})^{\gamma}-\int_{\mathbb{T}^{3}}%
(\rho^{\varepsilon})^{\gamma}\right)  \right)  , \label{flux_ani_1}%
\end{equation}
as an effective flux (of course when $\theta=0$, $F_{an}^{\varepsilon}$
coincides with $F^{\varepsilon}$ defined in $\left(  \text{\ref{flux}}\right)
$). This fails because we do not control the sign of
\[
\overline{\rho^{\theta}\left(  \int_{\mathbb{T}^{3}}\rho^{\gamma}%
+\Delta_{\theta}^{-1}\Delta\left(  (\rho^{\gamma}-\int_{\mathbb{T}^{3}}%
\rho^{\gamma}\right)  \right)  }-\overline{\rho^{\theta}}\overline{\left(
\int_{\mathbb{T}^{3}}\rho^{\gamma}+\Delta_{\theta}^{-1}\Delta\left(
(\rho^{\gamma}-\int_{\mathbb{T}^{3}}\rho^{\gamma}\right)  \right)  },
\]
as we do when $\theta=0$. Thus, in this case the equivalent of $\left(
\text{\ref{compacteness_iso}}\right)  $ is of no use for the identification of
$\rho^{\gamma}$ with $\overline{\rho^{\gamma}}$. Secondly, we could apply
$\operatorname{div}_{\theta}$, defined in $\left(  \text{\ref{div_theta}%
}\right)  $, in the momentum equation in order to obtain
\[
-\Delta_{\theta}\left(  \mu\operatorname{div}_{\theta}u^{\varepsilon}%
-(\rho^{\varepsilon})^{\gamma}\right)  =-\operatorname{div}_{\theta}\left(
\rho^{\varepsilon}u^{\varepsilon}\cdot\nabla u^{\varepsilon}\right)
+\operatorname{div}_{\theta}g^{\varepsilon},
\]
which yields compactness for the anisotropic effective-flux
\begin{equation}
\tilde{F}_{an}^{\varepsilon}=\mu\operatorname{div}_{\theta}u^{\varepsilon
}-(\rho^{\varepsilon})^{\gamma}. \label{flux_ani_2}%
\end{equation}
The problem is that this new quantity does not appear in the transport
equation such that we cannot use it in order to replace $\overline
{\rho^{\theta}\operatorname{div}u}$ with a more appropriate formula (unless,
of course, we would have more information on $\partial_{3}u^{3}$ which is not
the case). \smallskip The key ingredient in the proof of Theorem \ref{Main2}
turns out to be the fact that we can recover compactness properties for the
gradient of the velocity. In order to achieve this we have to use the
renormalized stationary transport equation and to also take into account the
momentum equation. More precisely the following proposition holds true:

\begin{proposition}
\label{Identity_1} Under the hypothesis of Theorem \ref{Main2}, the following
identity%
\begin{equation}
\frac{1}{\gamma-1}\operatorname{div}\left(  u\left(  \overline{\rho^{\gamma}%
}-\rho^{\gamma}\right)  \right)  +\left(  \overline{\rho^{\gamma}}%
-\rho^{\gamma}\right)  \operatorname{div}u+\overline{\mathcal{C}\left(
u,u\right)  }-\mathcal{C}\left(  u,u\right)  =0, \label{Identity}%
\end{equation}
holds true in the sense of distributions.
\end{proposition}

\begin{proof}
Owing to Proposition \ref{Prop_ren2} we get that
\begin{equation}
\operatorname{div}\left(  (\rho^{\varepsilon})^{\gamma}u^{\varepsilon}\right)
+\left(  \gamma-1\right)  (\rho^{\varepsilon})^{\gamma}\operatorname{div}%
u^{\varepsilon}=0. \label{renorm_eps}%
\end{equation}
The fact that $\left(  \rho^{\varepsilon},u^{\varepsilon}\right)  $ verify the
bounds $\left(  \text{\ref{uniform_bounds}}\right)  $ allows us to extend the
weak formulation of the velocity's equation to test functions $\psi$ for which
$\psi\in(L^{2}\left(  \mathbb{T}^{3}\right)  )^{3},\nabla\psi\in(L^{2}\left(
\mathbb{T}^{3}\right)  )^{9}$. \ Thus, taking $\varphi\in C^{\infty}\left(
\mathbb{T}^{3}\right)  $, may use $\varphi u$ as a test function in the weak
formulation of the velocity's equation and using $\left(
\text{\ref{Au_scalaire_u}}\right)  $ and $\left(  \text{\ref{renorm_eps}%
}\right)  $ we get that%
\[
\operatorname{div}\left(  (\rho^{\varepsilon})^{\gamma}u^{\varepsilon}\right)
=-\frac{\left(  \gamma-1\right)  }{\gamma}\left(  \frac{1}{2}%
\operatorname{div}\left(  \rho^{\varepsilon}u^{\varepsilon}\left\vert
u^{\varepsilon}\right\vert ^{2}\right)  -\mathcal{B}\left(  u^{\varepsilon
},u^{\varepsilon}\right)  +\mathcal{C}\left(  u^{\varepsilon},u^{\varepsilon
}\right)  +u^{\varepsilon}g^{\varepsilon}\right)  ,
\]
where $\mathcal{B},\mathcal{C}$ are defined by $\left(
\text{\ref{definition_B_et_C}}\right)  $. The convergence properties announced
in $\left(  \text{\ref{convergence_properties}}\right)  $ allow us to conclude
that%
\[
\operatorname{div}\left(  \overline{\rho^{\gamma}}u\right)  =-\frac{\left(
\gamma-1\right)  }{\gamma}\left(  \frac{1}{2}\operatorname{div}\left(  \rho
u\left\vert u\right\vert ^{2}\right)  -\mathcal{B}\left(  u,u\right)
+\overline{\mathcal{C}\left(  u,u\right)  }+ug\right)  .
\]
Of course, we can do the same manipulations to $\left(  \rho,u\right)  $ in
order to obtain that%
\begin{align*}
\operatorname{div}\left(  \rho^{\gamma}u\right)   &  =\frac{\left(
\gamma-1\right)  }{\gamma}\left\{  \operatorname{div}\left(  \left(
\overline{\rho^{\gamma}}-\rho^{\gamma}\right)  u\right)  -\left(
\overline{\rho^{\gamma}}-\rho^{\gamma}\right)  \operatorname{div}u\right\} \\
&  -\frac{\left(  \gamma-1\right)  }{\gamma}\left(  \frac{1}{2}%
\operatorname{div}\left(  \rho u\left\vert u\right\vert ^{2}\right)
-\mathcal{B}\left(  u,u\right)  +\mathcal{C}\left(  u,u\right)  +ug\right)  .
\end{align*}
Thus, by taking the difference we get $\left(  \text{\ref{Identity}}\right)  $
which ends the proof.
\end{proof}
Next, we claim that

\begin{proposition}
\label{Convergenta_gradient}Under the hypothesis of Theorem \ref{Main2}, we
have that%
\[
\nabla u^{\varepsilon}\rightarrow\nabla u\text{ strongly in }L^{\frac{3\left(
\gamma-1\right)  }{\gamma}}\left(  \mathbb{T}^{3}\right)  .
\]

\end{proposition}

\begin{proof}
This will result from the manipulation of the identity proved in Proposition
\ref{Identity_1}. Consider a regularizing kernel $\left(  \omega_{\alpha
}\right)  _{\alpha>0}$ and using $\left(  \text{\ref{Identity}}\right)  $ we
may write that
\[
\operatorname{div}\left(  u\delta_{\alpha}\right)  +\left(  \gamma-1\right)
\delta_{\alpha}\operatorname{div}u=-\left(  \gamma-1\right)  \omega_{\alpha
}\ast\left(  \overline{\mathcal{C}\left(  u,u\right)  }-\mathcal{C}\left(
u,u\right)  \right)  +r_{\alpha}\left(  u,\rho,\overline{\rho^{\gamma}%
}\right)
\]
where%
\[
\left\{
\begin{array}
[c]{l}%
\delta_{\alpha}=\omega_{\alpha}\ast\left(  \overline{\rho^{\gamma}}%
-\rho^{\gamma}\right)  ,\\
r_{\alpha}\left(  u,\rho,\overline{\rho^{\gamma}}\right)  =\operatorname{div}%
\left\{  u\omega_{\alpha}\ast\left(  \overline{\rho^{\gamma}}-\rho^{\gamma
}\right)  -\omega_{\alpha}\ast\left[  \left(  \overline{\rho^{\gamma}}%
-\rho^{\gamma}\right)  u\right]  \right\}  \\
+\left(  \gamma-1\right)  \left\{  \operatorname{div}u\omega_{\alpha}%
\ast\left(  \overline{\rho^{\gamma}}-\rho^{\gamma}\right)  -\omega_{\alpha
}\ast\left[  \left(  \overline{\rho^{\gamma}}-\rho^{\gamma}\right)
\operatorname{div}u\right]  \right\}  .
\end{array}
\right.
\]
Let $h>0$ be a constant and multiply the last equality with $\frac{1}{\gamma
}(\delta_{\alpha}+h)^{\frac{1}{\gamma}-1}$ in order to get that%
\begin{align*}
\operatorname{div}\left(  u(\delta_{\alpha}+h)^{\frac{1}{\gamma}}\right)   &
=-\frac{\gamma-1}{\gamma}(\delta_{\alpha}+h)^{\frac{1}{\gamma}-1}%
\omega_{\alpha}\ast\left(  \overline{\mathcal{C}\left(  u,u\right)
}-\mathcal{C}\left(  u,u\right)  \right)  \\
&  +(\delta_{\alpha}+h)^{\frac{1}{\gamma}-1}h\operatorname{div}u+\frac
{1}{\gamma}(\delta_{\alpha}+h)^{\frac{1}{\gamma}-1}r_{\alpha}\left(
u,\rho,\overline{\rho^{\gamma}}\right)  .
\end{align*}
Using Proposition \ref{Prop_ren1} we see that taking the limit $\alpha
\rightarrow0$ yields%
\begin{align*}
\operatorname{div}\left(  u\left(  \left(  \overline{\rho^{\gamma}}%
-\rho^{\gamma}\right)  +h\right)  ^{\frac{1}{\gamma}}\right)   &
=-\frac{\gamma-1}{\gamma}\left(  \overline{\mathcal{C}\left(  u,u\right)
}-\mathcal{C}\left(  u,u\right)  \right)  \left(  \left(  \overline
{\rho^{\gamma}}-\rho^{\gamma}\right)  +h\right)  ^{\frac{1}{\gamma}-1}\\
&  +\left(  \left(  \overline{\rho^{\gamma}}-\rho^{\gamma}\right)  +h\right)
^{\frac{1}{\gamma}-1}h\operatorname{div}u.
\end{align*}
Integrating the last equation gives us%
\begin{equation}
\int_{\mathbb{T}^{3}}\left(  \overline{\mathcal{C}\left(  u,u\right)
}-\mathcal{C}\left(  u,u\right)  \right)  \left(  \left(  \overline
{\rho^{\gamma}}-\rho^{\gamma}\right)  +h\right)  ^{\frac{1}{\gamma}-1}%
=\frac{\gamma h}{\gamma-1}\int_{\mathbb{T}^{3}}\left(  \left(  \overline
{\rho^{\gamma}}-\rho^{\gamma}\right)  +h\right)  ^{\frac{1}{\gamma}%
-1}\operatorname{div}u,\label{Integrala=0}%
\end{equation}
which can be put under the form%
\begin{align}
&  \int_{\left(  \overline{\rho^{\gamma}}=\rho^{\gamma}\right)  }\left(
\overline{\mathcal{C}\left(  u,u\right)  }-\mathcal{C}\left(  u,u\right)
\right)  +\int_{\left(  \overline{\rho^{\gamma}}\not =\rho^{\gamma}\right)
}\left(  \overline{\mathcal{C}\left(  u,u\right)  }-\mathcal{C}\left(
u,u\right)  \right)  \left(  \frac{h}{\left(  \overline{\rho^{\gamma}}%
-\rho^{\gamma}\right)  +h}\right)  ^{1-\frac{1}{\gamma}}\nonumber\\
&  =\frac{\gamma h}{\gamma-1}\int_{\mathbb{T}^{3}}\left(  \frac{h}{\left(
\overline{\rho^{\gamma}}-\rho^{\gamma}\right)  +h}\right)  ^{1-\frac{1}%
{\gamma}}\operatorname{div}u.\label{Integrala_1}%
\end{align}
Now, using that
\begin{equation}
\left\{
\begin{array}
[c]{l}%
\lim_{h\rightarrow0}\dfrac{h}{\left(  \overline{\rho^{\gamma}}-\rho^{\gamma
}\right)  +h}=0\text{ a.e. on }\left(  \overline{\rho^{\gamma}}\not =%
\rho^{\gamma}\right)  \text{ and }\\
\text{for any }h>0\text{ then }\dfrac{h}{\left(  \overline{\rho^{\gamma}}%
-\rho^{\gamma}\right)  +h}\leq1\text{ a.e. on }\mathbb{T}^{3},
\end{array}
\right.  \label{dominated}%
\end{equation}
we get that%
\[
\int_{\left(  \overline{\rho^{\gamma}}=\rho^{\gamma}\right)  }\left(
\overline{\mathcal{C}\left(  u,u\right)  }-\mathcal{C}\left(  u,u\right)
\right)  =0.
\]
As a consequence we get that%
\begin{equation}
\left(  \overline{\mathcal{C}\left(  u,u\right)  }-\mathcal{C}\left(
u,u\right)  \right)  \text{ a.e. on }\left(  \overline{\rho^{\gamma}}%
=\rho^{\gamma}\right)  .\label{recover_strongconv_on_u_1}%
\end{equation}
Then we see that $\left(  \text{\ref{Integrala_1}}\right)  $ rewrites%
\[
\int_{\left(  \overline{\rho^{\gamma}}\not =\rho^{\gamma}\right)  }\left(
\overline{\mathcal{C}\left(  u,u\right)  }-\mathcal{C}\left(  u,u\right)
\right)  \left(  \frac{h}{\left(  \overline{\rho^{\gamma}}-\rho^{\gamma
}\right)  +h}\right)  ^{1-\frac{1}{\gamma}}=\frac{\gamma h}{\gamma-1}%
\int_{\mathbb{T}^{3}}\left(  \frac{h}{\left(  \overline{\rho^{\gamma}}%
-\rho^{\gamma}\right)  +h}\right)  ^{1-\frac{1}{\gamma}}\operatorname{div}u,
\]
which we put under the form%
\[
\int_{\left(  \overline{\rho^{\gamma}}\not =\rho^{\gamma}\right)  }\left(
\overline{\mathcal{C}\left(  u,u\right)  }-\mathcal{C}\left(  u,u\right)
\right)  \left(  \frac{1}{\left(  \overline{\rho^{\gamma}}-\rho^{\gamma
}\right)  +h}\right)  ^{1-\frac{1}{\gamma}}=\frac{\gamma h^{\frac{1}{\gamma}}%
}{\gamma-1}\int_{\mathbb{T}^{3}}\left(  \frac{h}{\left(  \overline
{\rho^{\gamma}}-\rho^{\gamma}\right)  +h}\right)  ^{1-\frac{1}{\gamma}%
}\operatorname{div}u,
\]
such that using the inequality from $\left(  \text{\ref{dominated}}\right)  $
we get that%
\begin{equation}
\int_{\left(  \overline{\rho^{\gamma}}\not =\rho^{\gamma}\right)  }\left(
\overline{\mathcal{C}\left(  u,u\right)  }-\mathcal{C}\left(  u,u\right)
\right)  \left(  \frac{1}{\left(  \overline{\rho^{\gamma}}-\rho^{\gamma
}\right)  +h}\right)  ^{1-\frac{1}{\gamma}}\leq\frac{\gamma h^{\frac{1}%
{\gamma}}}{\gamma-1}\left\Vert \operatorname{div}u\right\Vert _{L^{1}%
}.\label{inegalite_div}%
\end{equation}
For all $n>0$ we have
\[
\left\{  x:\overline{\rho^{\gamma}}\left(  x\right)  \geq\rho^{\gamma}\left(
x\right)  +1/n\right\}  \subset\left\{  x:\overline{\rho^{\gamma}}\left(
x\right)  \not =\rho^{\gamma}\left(  x\right)  \right\}
\]
and as the integrand from the left hand side of the inequality $\left(
\text{\ref{inegalite_div}}\right)  $ is positive, we get that%
\[
\int_{\left(  \overline{\rho^{\gamma}}\geq\rho^{\gamma}+1/n\right)  }\left(
\overline{\mathcal{C}\left(  u,u\right)  }-\mathcal{C}\left(  u,u\right)
\right)  \left(  \frac{1}{\left(  \overline{\rho^{\gamma}}-\rho^{\gamma
}\right)  +h}\right)  ^{1-\frac{1}{\gamma}}\leq\frac{\gamma h^{\frac{1}%
{\gamma}}}{\gamma-1}\left\Vert \operatorname{div}u\right\Vert _{L^{1}}.
\]
Taking in account that
\begin{equation}
\left\{
\begin{array}
[c]{l}%
\lim_{h\rightarrow0}\dfrac{1}{\left(  \overline{\rho^{\gamma}}-\rho^{\gamma
}\right)  +h}=\dfrac{1}{\left(  \overline{\rho^{\gamma}}-\rho^{\gamma}\right)
}\text{ a.e. on }\left(  \overline{\rho^{\gamma}}\geq\rho^{\gamma}+\frac{1}%
{n}\right)  \text{ and }\\
\left(  \dfrac{1}{\left(  \overline{\rho^{\gamma}}-\rho^{\gamma}\right)
+h}\right)  ^{1-\frac{1}{\gamma}}\leq n^{1-\frac{1}{\gamma}}\text{ a.e. on
}\left(  \overline{\rho^{\gamma}}\geq\rho^{\gamma}+\frac{1}{n}\right)  ,
\end{array}
\right.  \label{dominated2}%
\end{equation}
we get via the dominated convergence theorem that
\[
\int_{\left(  \overline{\rho^{\gamma}}\geq\rho^{\gamma}+1/n\right)  }\left(
\overline{\mathcal{C}\left(  u,u\right)  }-\mathcal{C}\left(  u,u\right)
\right)  \left(  \overline{\rho^{\gamma}}-\rho^{\gamma}\right)  ^{\frac
{1}{\gamma}-1}=0,
\]
which yields
\begin{equation}
\overline{\mathcal{C}\left(  u,u\right)  }-\mathcal{C}\left(  u,u\right)
=0\text{ a.e. on }\left\{  x:\overline{\rho^{\gamma}}\left(  x\right)
\geq\rho^{\gamma}\left(  x\right)  +1/n\right\}
.\label{recover_strongconv_on_u_2}%
\end{equation}
As $n$ is arbitrary we deduce that
\begin{equation}
\overline{\mathcal{C}\left(  u,u\right)  }=\mathcal{C}\left(  u,u\right)
\text{ a.e. on }\left(  \overline{\rho^{\gamma}}>\rho^{\gamma}\right).
\label{egalitate_pe_stricta}%
\end{equation}
Putting together the two relations $\left(
\text{\ref{recover_strongconv_on_u_1}}\right)  $ and $\left(
\text{\ref{egalitate_pe_stricta}}\right)  $ we get that%
\[
\overline{\mathcal{C}\left(  u,u\right)  }-\mathcal{C}\left(  u,u\right)
=0\text{ a.e. on }\mathbb{T}^{3}%
\]
and consequently%
\[
\nabla u^{\varepsilon}\rightarrow\nabla u\text{ in }L^{r}\left(
\mathbb{T}^{3}\right)  ,
\]
for all $r\in\lbrack1,\frac{3\left(  \gamma-1\right)  }{\gamma})$. This
concludes the proof of Proposition \ref{Convergenta_gradient}.
\end{proof}\begin{proof}[End of proof of Theorem \ref{Main2}.] The fact that $(\nabla
u^{\varepsilon})_{\varepsilon>0}$ converges strongly to $\nabla u$ along with
the fact that the anisotropic effective flux is compact will be used to
identify $\overline{\rho^{\gamma}}$ with $\rho^{\gamma}$. Indeed, let us
observe that owing to $\left(  \text{\ref{div_theta_Au}}\right)  $, when
applying $\operatorname{div}_{\theta}$ in the second equation of $\left(
\text{\ref{system}}\right)  $ we obtain that%
\[
-\Delta_{\theta}(\mu\operatorname{div}_{\theta}u^{\varepsilon}+\left(
\mu+\lambda\right)  \operatorname{div}u^{\varepsilon}+\xi\ast
\operatorname{div}u^{\varepsilon}-(\rho^{\varepsilon})^{\gamma})=\Delta
(\eta\ast\operatorname{div}_{\theta}u)-\operatorname{div}_{\theta}\left(
\rho^{\varepsilon}u^{\varepsilon}\cdot\nabla u^{\varepsilon}\right)
+\operatorname{div}_{\theta}g^{\varepsilon}%
\]
such that%
\begin{align*}
\nabla\left(  \mu\operatorname{div}_{\theta}u^{\varepsilon}+\left(
\mu+\lambda\right)  \operatorname{div}u^{\varepsilon}+\xi\ast
\operatorname{div}u^{\varepsilon}-(\rho^{\varepsilon})^{\gamma}\right)   &
=-\left(  -\Delta_{\theta}\right)  ^{-1}(-\Delta)(\eta\ast\operatorname{div}%
_{\theta}u)\\
&  +\left(  -\Delta_{\theta}\right)  ^{-1}\nabla\left(  -\operatorname{div}%
_{\theta}\left(  \rho^{\varepsilon}u^{\varepsilon}\cdot\nabla u^{\varepsilon
}\right)  +\operatorname{div}_{\theta}g^{\varepsilon}\right)
\end{align*}
and we recover that%
\[
\mu\operatorname{div}_{\theta}u^{\varepsilon}+\left(  \mu+\lambda\right)
\operatorname{div}u^{\varepsilon}+\xi\ast\operatorname{div}u^{\varepsilon
}-(\rho^{\varepsilon})^{\gamma}\in W^{1,\frac{3(\gamma-1)}{2\gamma-1}}\left(
\mathbb{T}^{3}\right)
\]
and therefore, owing to the Rellich-Kondrachov we get that%
\[
\lim_{\varepsilon\rightarrow0}(\mu\operatorname{div}_{\theta}u^{\varepsilon
}+\left(  \mu+\lambda\right)  \operatorname{div}u^{\varepsilon}+\xi
\ast\operatorname{div}u^{\varepsilon}-(\rho^{\varepsilon})^{\gamma}%
)=\mu\operatorname{div}_{\theta}u+\left(  \mu+\lambda\right)
\operatorname{div}u+\xi\ast\operatorname{div}u-\overline{\rho^{\gamma}}%
\]
strongly in $L^{r}\left(  \mathbb{T}^{3}\right)  $ for all $r\in\lbrack
1,\frac{3\left(  \gamma-1\right)  }{\gamma})$. This implies that%
\begin{align*}
&  \lim_{\varepsilon\rightarrow0}\rho^{\varepsilon}\left(  \mu
\operatorname{div}_{\theta}u^{\varepsilon}+\left(  \mu+\lambda\right)
\operatorname{div}u^{\varepsilon}+\xi\ast\operatorname{div}u^{\varepsilon
}-(\rho^{\varepsilon})^{\gamma}\right) \\
&  =\rho\left(  \mu\operatorname{div}_{\theta}u+\left(  \mu+\lambda\right)
\operatorname{div}u+\xi\ast\operatorname{div}u-\overline{\rho^{\gamma}%
}\right)  \text{ weakly in }L^{r}\left(  \mathbb{T}^{3}\right)  ,
\end{align*}
for some $r>1$. Of course, we may use the strong convergence of $\nabla
u^{\varepsilon}$ to $\nabla u$ in order to conclude that%
\begin{align*}
&  \lim_{\varepsilon\rightarrow0}\rho^{\varepsilon}\left(  \mu
\operatorname{div}_{\theta}u^{\varepsilon}+\left(  \mu+\lambda\right)
\operatorname{div}u^{\varepsilon}+\xi\ast\operatorname{div}u^{\varepsilon
}\right) \\
&  =\rho\left(  \mu\operatorname{div}_{\theta}u+\left(  \mu+\lambda\right)
\operatorname{div}u+\xi\ast\operatorname{div}u\right)  \text{ weakly in }%
L^{r}\left(  \mathbb{T}^{3}\right)  ,
\end{align*}
for some $r>1$. Combining the last two identities we get that%
\[
\lim_{\varepsilon\rightarrow0}(\rho^{\varepsilon})^{\gamma+1}=\rho
\overline{\rho^{\gamma}}\text{ weakly in }L^{r}\left(  \mathbb{T}^{3}\right)
,
\]
with $r>1$ which, of course, implies that $\rho^{\gamma}=\overline
{\rho^{\gamma}}$. This concludes the proof of Theorem \ref{Main2}.
\end{proof}

\section{Construction of approximate solutions \label{CAS}}

A weak solution for system
\eqref{Stationary_Stokes}--\eqref{diffusion_operator} will be obtained as the
limit of solutions of the following regularized system%
\begin{equation}
\left\{
\begin{array}
[c]{l}%
-\varepsilon\Delta\rho+\delta\left(  \rho-M\right)  +\operatorname{div}\left(
\rho\omega_{\delta}\ast u\right)  =0,\\
\dfrac{\delta}{2}\left(  \rho u-%
{\displaystyle\int_{\mathbb{T}^{3}}}
\rho u\right)  +\operatorname{div}\left(  \rho\omega_{\delta}\ast u\otimes
u\right)  -\mathcal{A}u+\nabla\left(  \omega_{\delta}\ast\rho^{\gamma}\right)
+\varepsilon\left(  \nabla u\nabla\rho-%
{\displaystyle\int_{\mathbb{T}^{3}}}
\nabla u\nabla\rho\right)  =\omega_{\delta}\ast g,\\
\rho\geq0,\text{ }%
{\displaystyle\int_{\mathbb{T}^{3}}}
\rho=M,\text{ }%
{\displaystyle\int_{\mathbb{T}^{3}}}
u=0,
\end{array}
\right.  \label{Approximare_eps_delta}%
\end{equation}
when the regularization parameters $\delta,\varepsilon\in\left(  0,1\right)
^{2}$ tend to $0$. Above,
\[
\omega_{\delta}\left(  \cdot\right)  =\frac{1}{\delta^{3}}\omega\left(
\frac{1}{\delta}\cdot\right)
\]
with $\omega\in\mathcal{D}\left(  \mathbb{R}^{3}\right)  $ a smooth,
non-negative, even function which is compactly supported in the unit ball
centered at the origin and with integral $1$. The fact that we can solve the
above system is a consequence of the Leray-Schauder fixed point theorem, see
Theorem \ref{Schauder_Leray} from the Appendix. The objective of the next
section is to construct solutions for $\left(
\text{\ref{Approximare_eps_delta}}\right)  $.

\subsection{Existence of solutions for the approximate system $\left(
\text{\ref{Approximare_eps_delta}}\right)  $%
\label{Section_Existence_for_regularized1}}

Let us fix $\left(  \varepsilon,\delta\right)  \in\left(  0,1\right)  ^{2}$.
We begin by the following proposition.

\begin{proposition}
\label{transport_stationar}Consider $v\in\left(  W^{1,\infty}\left(
\mathbb{T}^{3}\right)  \right)  ^{3}$ and $M,\delta,\varepsilon>0$. Then there
exists a unique positive solution $\rho\in W^{2,2}\left(  \mathbb{T}%
^{3}\right)  $ for the equation%
\[
-\varepsilon\Delta\rho+\delta\left(  \rho-M\right)  +\operatorname{div}\left(
\rho v\right)  =0.
\]
Moreover, there exists a positive constant $C\left(  M,\varepsilon\right)  $
depending on $\varepsilon$ and $M$ such that:%
\[
\left\Vert \rho\right\Vert _{W^{2,2}}\leq C\left(  M,\varepsilon\right)
\left(  1+\left\Vert v\right\Vert _{W^{1,\infty}}^{2}\right)  .
\]

\end{proposition}

\begin{proof}
The proof is a classical application of the Leray-Schauder theorem. For any
$r\in W^{1,2}\left(  \mathbb{T}^{3}\right)  $ we consider $T\left(  r\right)
\in W^{1,2}\left(  \mathbb{T}^{3}\right)  $ verifying
\begin{equation}
-\varepsilon\Delta T\left(  r\right)  +\delta\left(  T\left(  r\right)
-M\right)  +\operatorname{div}\left(  rv\right)  =0.\label{definition_T}%
\end{equation}
The existence of $T\left(  r,v\right)  \in W^{1,2}\left(  \mathbb{T}%
^{3}\right)  $ is a consequence of the Lax-Milgram theorem.
\medskip \noindent\textit{Continuity and compactness of the operator $T$.} Observe
that
\begin{equation}
\int_{\mathbb{T}^{3}}T\left(  r\right)  dx=M.\label{interala_M}%
\end{equation}
Using $\left(  \text{\ref{interala_M}}\right)  $ we have that%
\begin{align*}
\varepsilon\int_{\mathbb{T}^{3}}\left\vert \nabla T\left(  r\right)
\right\vert ^{2}+\delta\int_{\mathbb{T}^{3}}\left\vert T\left(  r\right)
\right\vert ^{2} &  \leq\delta M^{2}+\left\Vert r\right\Vert _{L^{6}%
}\left\Vert v\right\Vert _{L^{3}}\left\Vert \nabla T\left(  r\right)
\right\Vert _{L^{2}}\\
&  \leq M^{2}+\frac{1}{2\varepsilon}\left\Vert r\right\Vert _{L^{6}}%
^{2}\left\Vert v\right\Vert _{L^{\infty}}^{2}+\frac{\varepsilon}{2}\left\Vert
\nabla T\left(  r\right)  \right\Vert _{L^{2}}^{2},
\end{align*}
which gives us%
\begin{equation}
\varepsilon^{\frac{1}{2}}\left\Vert \nabla T\left(  r\right)  \right\Vert
_{L^{2}}+\delta^{\frac{1}{2}}\left\Vert T\left(  r\right)  \right\Vert
_{L^{2}}\leq M+C\left(  \varepsilon\right)  \left\Vert v\right\Vert
_{L^{\infty}}\left\Vert r\right\Vert _{W^{1,2}}%
.\label{estimare_independenta_delta_1}%
\end{equation}
\bigskip We also have that%
\begin{align}
\varepsilon\left\Vert \Delta T\left(  r\right)  \right\Vert _{L^{2}} &
\leq\delta  \left\Vert T\left(  r\right)  -M\right\Vert _{L^{2}}
+\left\Vert r\right\Vert _{L^{2}}\left\Vert \operatorname{div}v\right\Vert
_{L^{\infty}}+\left\Vert v\right\Vert _{L^{\infty}}\left\Vert \nabla
r\right\Vert _{L^{2}}\nonumber\\
&  \leq C\left(  M,\varepsilon\right)  \left(  1+\left\Vert v\right\Vert
_{W^{1,\infty}}\left\Vert r\right\Vert _{W^{1,2}}\right)
,\label{estimare_Delta_T}%
\end{align}
Consequently
\begin{equation}
T\left(  r\right)  \in W^{2,2}\left(  \mathbb{T}^{3}\right)
,\label{compactness_in_rho}%
\end{equation}
such that using the Sobolev inequality, one also has that $\nabla T\left(
r\right)  \in L^{6}\left(  \mathbb{T}^{3}\right)  $ and $T\left(  r\right)
\in L^{r}\left(  \mathbb{T}^{3}\right)  $ for all $r\in\lbrack1,\infty]$ with%
\begin{equation}
\left\Vert \nabla T\left(  r\right)  \right\Vert _{L^{6}}+\left\Vert T\left(
r\right)  \right\Vert _{L^{r}}\leq C\left(  M,\varepsilon\right)  \left(
1+\left\Vert v\right\Vert _{W^{1,\infty}}\left\Vert r\right\Vert _{W^{1,2}%
}\right)  .\label{estimari_gradient_si_sup}%
\end{equation}
Next, consider $r_{0}\in W^{1,2}\left(  \mathbb{T}^{3}\right)  $ and $r\in
W^{1,2}\left(  \mathbb{T}^{3}\right)  $ such that%
\[
\left\Vert r-r_{0}\right\Vert _{W^{1,2}}\leq1.
\]
First, we see that%
\[
-\varepsilon\Delta\left(  T\left(  r\right)  -T\left(  r_{0}\right)  \right)
+\delta\left(  T\left(  r\right)  -T(r_{0})\right)  +\operatorname{div}\left(
(r-r_{0})v\right)  =0,
\]
such that multiplying with $T\left(  r\right)  -T\left(  r_{0}\right)  $ and
applying Cauchy's inequality gives us%
\begin{equation}
\left\Vert T\left(  r\right)  -T\left(  r_{0}\right)  \right\Vert _{W^{1,2}%
}\leq C\left(  M,\varepsilon\right)  \left\Vert v\right\Vert _{L^{\infty}%
}\left\Vert r-r_{0}\right\Vert _{L^{2}}\leq C\left(  M,\varepsilon\right)
\left\Vert v\right\Vert _{L^{\infty}}\left\Vert r-r_{0}\right\Vert _{W^{1,2}}.
\end{equation}
Thus, $T$ is a continuous operator from $W^{1,2}\left(  \mathbb{T}^{3}\right)
$ into itself which is compact in view of $\left(
\text{\ref{compactness_in_rho}}\right)  $. To apply the Leray-Schauder fixed point theorem, we define the set
\[
\mathcal{P=}\left\{  \rho\in W^{1,2}\left(  \mathbb{T}^{3}\right)
:\rho=\lambda T(\rho)\text{ for some }\lambda\in(0,1]\right\}  .
\]
and prove that it is a bounded set.
\noindent\textit{The set $\mathcal{P}$ is bounded.}
Consider $\rho$ $\in\mathcal{P}$ and $\lambda\in(0,1]$ such that%
\begin{equation}
-\varepsilon\Delta\rho+\delta\left(  \rho-\lambda M\right)  +\lambda
\operatorname{div}\left(  \rho v\right)  =0.\label{equation_rho_lambda}%
\end{equation}
We begin by proving that that such a $\rho$ is positive. In order to achieve
this, consider
\[
\psi_{\eta}\left(  s\right)  =\frac{\sqrt{\eta+s^{2}}-s}{2}%
\]
which is smooth and verifies for all $s\in\mathbb{R}$ and $\eta>0$%
\begin{equation}
\left\{
\begin{array}
[c]{l}%
0\leq\psi_{\eta}\left(  s\right)  -\left(  \dfrac{\left\vert s\right\vert
-s}{2}\right)  \leq\dfrac{\sqrt{\eta}}{2},\\
\psi_{\eta}^{\prime}\left(  s\right)  \leq0,\psi_{\eta}^{\prime\prime}\left(
s\right)  \geq0,\\
0\leq\psi_{\eta}\left(  s\right)  -s\psi_{\eta}^{\prime}\left(  s\right)
\leq\dfrac{\sqrt{\eta}}{2}.
\end{array}
\right.  \label{negative_part_approx}%
\end{equation}
Moreover, one can justify by regularization that for all $\eta>0$%
\begin{align*}
&  \delta\left(  \rho\psi_{\eta}^{\prime}\left(  \rho\right)  -\lambda
M\psi_{\eta}^{\prime}\left(  \rho\right)  \right)  +\lambda\operatorname{div}%
\left(  \psi_{\eta}\left(  \rho\right)  v\right)  +\lambda\left(  \rho
\psi_{\eta}^{\prime}\left(  \rho\right)  -\psi_{\eta}\left(  \rho\right)
\right)  \operatorname{div}v\\
&  =\varepsilon\Delta\psi_{\eta}\left(  \rho\right)  -\varepsilon\psi_{\eta
}^{^{\prime\prime}}\left(  \rho\right)  \left\vert \nabla\rho\right\vert ^{2}.
\end{align*}
We rewrite the last equation under the form%
\begin{align*}
\delta\psi_{\eta}\left(  \rho\right)   &  =\delta\left(  \psi_{\eta}\left(
\rho\right)  -\rho\psi_{\eta}^{\prime}\left(  \rho\right)  \right)  +\lambda
M\delta\psi_{\eta}^{\prime}\left(  \rho\right)  \\
&  -\lambda\operatorname{div}\left(  \psi_{\eta}\left(  \rho\right)  v\right)
-\lambda\left(  \rho\psi_{\eta}^{\prime}\left(  \rho\right)  -\psi_{\eta
}\left(  \rho\right)  \right)  \operatorname{div}v\\
&  +\varepsilon\Delta\psi_{\eta}\left(  \rho\right)  -\varepsilon\psi_{\eta
}^{^{\prime\prime}}\left(  \rho\right)  \left\vert \nabla\rho\right\vert ^{2}.
\end{align*}
By integration and using $\left(  \text{\ref{negative_part_approx}}\right)  $
we end up with%
\[
\delta\int_{\mathbb{T}^{3}}\psi_{\eta}(\rho)\leq\delta\int_{\mathbb{T}^{3}%
}\left(  \psi_{\eta}(\rho)-\rho\psi_{\eta}^{\prime}\left(  \rho\right)
\right)  -\lambda\int_{\mathbb{T}^{3}}\left(  \rho\psi_{\eta}^{\prime}\left(
\rho\right)  -\psi_{\eta}\left(  \rho\right)  \right)  \operatorname{div}v,
\]
which gives when $\eta\rightarrow0$
\[
\frac{\delta}{2}\int_{\mathbb{T}^{3}}\left(  \left\vert \rho\right\vert
-\rho\right)  \leq0,
\]
which implies that
\[
\rho\left(  x\right)  \geq0\text{ a.e. on }\mathbb{T}^{3}.
\]
Next, we see that by integrating $\left(  \text{\ref{equation_rho_lambda}%
}\right)  $ we get that%
\[
\left\Vert \rho\right\Vert _{L^{1}}=\int_{\mathbb{T}^{3}}\rho=\lambda M\leq M.
\]
We also have that%
\begin{align*}
\varepsilon\int_{\mathbb{T}^{3}}\left\vert \nabla\rho\right\vert ^{2}%
+\delta\int_{\mathbb{T}^{3}}\rho^{2} &  =\delta\lambda M^{2}+\lambda\int
\rho^{2}\operatorname{div}v\leq\delta M^{2}+\lambda\left\Vert
\operatorname{div}v\right\Vert _{L^{\infty}}\left\Vert \rho\right\Vert
_{L^{2}}^{2}\\
&  \leq\delta M^{2}+\left\Vert \operatorname{div}v\right\Vert _{L^{\infty}%
}\left\Vert \rho\right\Vert _{L^{2}}^{2}\leq\delta M^{2}+\left\Vert
\operatorname{div}v\right\Vert _{L^{\infty}}\left\Vert \rho\right\Vert
_{L^{1}}^{\frac{4}{5}}\left\Vert \rho\right\Vert _{L^{6}}^{\frac{12}{5}}\\
&  \leq\delta M^{2}+\frac{1}{4\alpha\varepsilon}M^{\frac{8}{5}}\left\Vert
\operatorname{div}v\right\Vert _{L^{\infty}}^{2}+\alpha\varepsilon\left\Vert
\rho-M\right\Vert _{L^{6}}^{2}\\
&  \leq\delta M^{2}+\frac{1}{4\alpha\varepsilon}M^{\frac{8}{5}}\left\Vert
\operatorname{div}v\right\Vert _{L^{\infty}}^{2}+\alpha\varepsilon C\left\Vert
\nabla\rho\right\Vert _{L^{2}}^{2}.
\end{align*}
We see that choosing $\alpha$ sufficiently small gives us
\[
\varepsilon\int_{\mathbb{T}^{3}}\left\vert \nabla\rho\right\vert ^{2}%
+\delta\int_{\mathbb{T}^{3}}\rho^{2}\leq C\left(  M,\varepsilon\right)
\left(  1+\left\Vert v\right\Vert _{W^{1,\infty}}^{2}\right)
\]
which means that the set is bounded.

\noindent{\textit{Existence of solution to the nonlinear equation.}}

Thanks to the Leray--Schauder Theorem, see Theorem
\ref{Schauder_Leray}, we get the existence of a fixed point for the operator
$T$ defined by $\left(  \text{\ref{definition_T}}\right)  $\ which obviously,
satisfies the equation%
\begin{equation}
-\varepsilon\Delta\rho+\delta\left(  \rho-M\right)  +\operatorname{div}\left(
\rho v\right)  =0.\label{system_rho}%
\end{equation}
and, moreover, verifies%
\[
\left\{
\begin{array}
[c]{l}%
{\displaystyle\int_{\mathbb{T}^{3}}}
\rho=M,\\
\varepsilon%
{\displaystyle\int_{\mathbb{T}^{3}}}
\left\vert \nabla\rho\right\vert ^{2}+\delta%
{\displaystyle\int_{\mathbb{T}^{3}}}
\rho^{2}\leq C\left(  M,\varepsilon\right)  \left(  1+\left\Vert v\right\Vert
_{W^{1,\infty}}^{2}\right)  .
\end{array}
\right.
\]
We also have that%
\[
\varepsilon\left\Vert \Delta\rho\right\Vert _{L^{2}}\leq\delta\left\Vert
\rho-M\right\Vert _{L^{2}}+\left\Vert \operatorname{div}\left(  \rho v\right)
\right\Vert _{L^{2}}\leq\left(  1+\left\Vert v\right\Vert _{W^{1,\infty}%
}\right)  \left\Vert \rho\right\Vert _{W^{1,2}}%
\]
which leads to
\[
\left\Vert \rho\right\Vert _{W^{2,2}}\leq C\left(  M,\varepsilon\right)
\left(  1+\left\Vert v\right\Vert _{W^{1,\infty}}^{2}\right)  .
\]
Let us consider two solutions $\left(  \rho,\tilde{\rho}\right)  $ of $\left(
\text{\ref{system_rho}}\right)  $ and observe that their difference verifies%
\[
-\varepsilon\Delta\left(  \rho-\tilde{\rho}\right)  +\delta\left(  \rho
-\tilde{\rho}\right)  +\operatorname{div}\left(  \left(  \rho-\tilde{\rho
}\right)  v\right)  =0.
\]
For all $\eta>0$, multiplying the above equation with $\varphi_{\eta}^{\prime
}\left(  \rho-\tilde{\rho}\right)  $ where
\[
\varphi_{\eta}\left(  s\right)  =\sqrt{\eta+s^{2}},
\]
integrating and making $\eta\rightarrow0$ we get that%
\[
\int_{\mathbb{T}^{3}}\left\vert \rho-\tilde{\rho}\right\vert =0.
\]
The details are left as exercise for the reader. This ends the proof of
Proposition \ref{transport_stationar}.
\end{proof}As was announced above, solutions for $\left(
\text{\ref{Approximare_eps_delta}}\right)  $ are obtained as fixed points of
an operator that is constructed in the following lines. We fix
\begin{equation}
M>0,\qquad g\in(L^{\frac{3\left(  \gamma-1\right)  }{2\gamma-1}}\left(
\mathbb{T}^{3}\right)  )^{3}\text{ with }\int_{\mathbb{T}^{3}}%
g=0,\label{M_et_g}%
\end{equation}
along with $\left(  \varepsilon,\delta\right)  \in\left(  0,1\right)  ^{2}$,
and for any $v\in\left(  W^{1,2}\left(  \mathbb{T}^{3}\right)  \right)  ^{3}$
with $%
{\displaystyle\int_{\mathbb{T}^{3}}}
v=0$ we consider $S\left(  v\right)  \in\left(  W^{1,2}\left(  \mathbb{T}%
^{3}\right)  \right)  ^{3}$ with $%
{\displaystyle\int_{\mathbb{T}^{3}}}
S\left(  v\right)  =0$ verifying
\begin{align}
-\mathcal{A}S\left(  v\right)   &  =-\frac{\delta}{2}\left(  \rho
v-\int_{\mathbb{T}^{3}}\rho v\right)  -\operatorname{div}\left(  \rho
\omega_{\delta}\ast v\otimes v\right)  -\nabla\left(  \omega_{\delta}\ast
\rho^{\gamma}\right)  \nonumber\\
&  -\varepsilon\left(  \nabla v\nabla\rho-\int_{\mathbb{T}^{3}}\nabla
v\nabla\rho\right)  +\omega_{\delta}\ast g,\label{definitie_S_}%
\end{align}
where $\rho\in W^{2,2}\left(  \mathbb{T}^{3}\right)  $ is the unique solution
of%
\[
-\varepsilon\Delta\rho+\delta\left(  \rho-M\right)  +\operatorname{div}\left(
\rho\omega_{\delta}\ast v\right)  =0.
\]
The existence of $S\left(  v\right)  $ is a consequence of the Lax-Milgram
theorem applied in the closed subspace of $\left(  W^{1,2}\left(
\mathbb{T}^{3}\right)  \right)  ^{3}$ of vector fields with zero mean. It
remains now to prove that we have a fixed point to solve the nonlinear
approximate system. This is the object of the following proposition.

\begin{proposition}
\label{Point_fix}The operator $S$ defined by $\left(  \text{\ref{definitie_S_}%
}\right)  $ admits a fixed point.
\end{proposition}

\begin{proof}
Proposition \ref{Point_fix} is a consequence of the Schauder-Leray theorem
(see Theorem \ref{Schauder_Leray}). We will first prove that $S$ is continuous
and compact and in a second time that the set%
\begin{equation}
\mathcal{P=}\left\{  u\in\left(  W^{1,2}\left(  \mathbb{T}^{3}\right)
\right)  ^{3}:u=\lambda S(u)\text{ for some }\lambda\in(0,1]\right\}
\label{set_P}%
\end{equation}
is bounded.
\smallskip\noindent1) \textit{Continuity and compactness of theoperator }$S$. First, let us recall that%
\[
-\varepsilon\Delta\rho+\delta\left(  \rho-M\right)  +\operatorname{div}\left(
\rho\omega_{\delta}\ast v\right)  =0,
\]
then%
\begin{equation}
\left\Vert \rho\right\Vert _{W^{2,2}}\leq C\left(  M,\varepsilon
,\delta\right)  \left(  1+\left\Vert v\right\Vert _{W^{1,2}}^{2}\right),
\label{Nouvelle_estimation_reg_rho}%
\end{equation}
see Proposition \ref{transport_stationar}. In the following lines we show
that $S\left(  v\right)  $ is actually more regular than $v$. We begin with%
\begin{align*}
&  \left\Vert \operatorname{div}\left(  \rho\omega_{\delta}\ast v\otimes
v\right)  \right\Vert _{L^{\frac{3}{2}}}\\
&  \leq\left\Vert \operatorname{div}\left(  \rho\omega_{\delta}\ast v\right)
v\right\Vert _{L^{\frac{3}{2}}}+\left\Vert (\rho\omega_{\delta}\ast
v)\cdot\nabla v\right\Vert _{L^{\frac{3}{2}}}\\
&  \leq\left\Vert \omega_{\delta}\ast v\right\Vert _{L^{\infty}}\left\Vert
\nabla\rho\right\Vert _{L^{2}}\left\Vert v\right\Vert _{L^{6}}+\left\Vert
\omega_{\delta}\ast\operatorname{div}v\right\Vert _{L^{\infty}}\left\Vert
\rho\right\Vert _{L^{2}}\left\Vert v\right\Vert _{L^{6}}+\left\Vert
\rho\right\Vert _{L^{6}}\left\Vert \omega_{\delta}\ast v\right\Vert
_{L^{\infty}}\left\Vert \nabla v\right\Vert _{L^{2}}\\
&  \leq C\left(  M,\varepsilon,\delta\right)  \left(  1+\left\Vert
v\right\Vert _{W^{1,2}}^{4}\right).
\end{align*}
Using $\left(  \text{\ref{Nouvelle_estimation_reg_rho}}\right)  $ we arrive at%
\[
\left\Vert \omega_{\delta}\ast g\right\Vert _{L^{\infty}}+\left\Vert
\nabla(\omega_{\delta}\ast\rho^{\gamma})\right\Vert _{L^{\infty}}\leq
C(M,\varepsilon,\delta,\left\Vert v\right\Vert _{W^{1,2}},\left\Vert
g\right\Vert _{L^{\frac{6}{5}}}).
\]
Using again $\left(  \text{\ref{estimari_gradient_si_sup}}\right)  $ we have
that%
\[
\left\Vert \nabla v\nabla\rho-%
{\displaystyle\int_{\mathbb{T}^{3}}}
\nabla v\nabla\rho\right\Vert _{L^{\frac{3}{2}}}\leq2\left\Vert \nabla
v\right\Vert _{L^{2}}\left\Vert \nabla\rho\right\Vert _{L^{6}}\leq2\left\Vert
\nabla v\right\Vert _{L^{2}}\left\Vert \rho\right\Vert _{W^{2,2}}\leq C\left(
M,\varepsilon,\delta\right)  \left(  1+\left\Vert v\right\Vert _{W^{1,2}}%
^{3}\right)  .
\]
Gathering the last three inequalities we get that%
\begin{equation}
\left\{
\begin{array}
[c]{l}%
\operatorname{div}\left(  \rho\omega_{\delta}\ast v\otimes v\right)
=\operatorname{div}\left(  \rho\omega_{\delta}\ast v\right)  v+(\rho
\omega_{\delta}\ast v)\cdot\nabla v\in(L^{\frac{3}{2}}\left(  \mathbb{T}%
^{3}\right)  )^{3},\\
\nabla(\omega_{\delta}\ast\rho^{\gamma}),\omega_{\delta}\ast g\in(L^{\infty
}\left(  \mathbb{T}^{3}\right)  )^{3},\\
\varepsilon\left(  \nabla v\nabla\rho-%
{\displaystyle\int_{\mathbb{T}^{3}}}
\nabla v\nabla\rho\right)  ,\dfrac{\delta}{2}\left(  \rho v-%
{\displaystyle\int_{\mathbb{T}^{3}}}
\rho v\right)  \in(L^{\frac{3}{2}}\left(  \mathbb{T}^{3}\right)  )^{3},
\end{array}
\right.  \label{terms_in_u}%
\end{equation}
which implies that%
\[
\mathcal{A}S\left(  v\right)  \in(L^{\frac{3}{2}}\left(  \mathbb{T}%
^{3}\right)  )^{3}%
\]
and consequently we get that%
\begin{equation}
S\left(  v\right)  \in(W^{2,\frac{3}{2}}\left(  \mathbb{T}^{3}\right)
)^{3}.\label{compactness_in_u}%
\end{equation}
Of course, sequences bounded in $W^{2,\frac{3}{2}}\left(  \mathbb{T}%
^{3}\right)  $ are precompact in $W^{1,2}\left(  \mathbb{T}^{3}\right)  $.\newline
Consider $\left(  v_{0},v_{1}\right)  \in\left(  W^{1,2}\left(  \mathbb{T}%
^{3}\right)  \right)  ^{3}\times\left(  W^{1,2}\left(  \mathbb{T}^{3}\right)
\right)  ^{3}$ such that
\[
\left\Vert v_{1}-v_{0}\right\Vert _{W^{1,2}}\leq1.
\]
Also, for $i\in\left\{  0,1\right\}  $ consider%
\[
-\varepsilon\Delta\rho_{i}+\delta\left(  \rho_{i}-M\right)
+\operatorname{div}\left(  \rho_{i}\omega_{\delta}\ast v_{i}\right)  =0
\]
and%
\begin{align*}
-\mathcal{A}S\left(  v_{i}\right)   &  =-\dfrac{\delta}{2}\left(  \rho
_{i}v_{i}-\int_{\mathbb{T}^{3}}\rho_{i}v_{i}\right)  -\operatorname{div}%
\left(  \rho_{i}\omega_{\delta}\ast v_{i}\otimes v_{i}\right)  -\nabla\left(
\omega_{\delta}\ast\rho_{i}^{\gamma}\right)  \\
&  -\varepsilon\left(  \nabla v_{i}\nabla\rho_{i}-\int\nabla v_{i}\nabla
\rho_{i}\right)  +\omega_{\delta}\ast g.
\end{align*}
First of all, the estimates $\left(  \text{\ref{Nouvelle_estimation_reg_rho}%
}\right)  $ allow us to conclude that%
\[
\left\Vert \rho_{0}\right\Vert _{W^{2,2}}+\left\Vert \rho_{1}\right\Vert
_{W^{2,2}}\leq C\left(  M,\varepsilon,\delta\right)  .
\]
The difference $\left(  \rho_{1}-\rho_{0}\right)  $ verifies%
\begin{equation}
-\varepsilon\Delta\left(  \rho_{1}-\rho_{0}\right)  +\delta\left(  \left(
\rho_{1}-\rho_{0}\right)  \right)  +\operatorname{div}\left(  \rho_{1}%
\omega_{\delta}\ast v_{1}-\rho_{0}\omega_{\delta}\ast v_{0}\right)
=0,\label{diferenta}%
\end{equation}
which provides the following estimate:%
\begin{equation}
\delta \int_{\mathbb{T}^{3}}\left\vert \rho_{1}-\rho_{0}\right\vert \leq
\int_{\mathbb{T}^{3}}\left\vert \operatorname{div}\left(  \rho_{0}\left(
v_{1}-v_{0}\right)  \right)  \right\vert \leq\left\Vert \rho_{0}\right\Vert
_{W^{1,2}}\left\Vert v_{1}-v_{0}\right\Vert _{W^{1,2}}%
.\label{dependenta_norma_L1}%
\end{equation}
Next, we see that%
\begin{align}
\left\Vert \rho_{1}-\rho_{0}\right\Vert _{W^{1,2}} &  \leq C\left\Vert
\rho_{1}\omega_{\delta}\ast v_{1}-\rho_{0}\omega_{\delta}\ast v_{0}\right\Vert
_{L^{2}}\nonumber\\
&  \leq C\left\Vert \omega_{\delta}\ast v_{1}\right\Vert _{L^{\infty}%
}\left\Vert \rho_{1}-\rho_{0}\right\Vert _{L^{2}}+C\left\Vert \rho
_{0}\right\Vert _{L^{6}}\left\Vert \omega_{\delta}\ast v_{1}-\omega_{\delta
}\ast v_{0}\right\Vert _{L^{3}}\nonumber\\
&  \leq C\left(  M,\varepsilon,\delta\right)  \left\Vert \rho_{1}-\rho
_{0}\right\Vert _{L^{1}}^{\frac{2}{5}}+C\left(  M,\varepsilon,\delta\right)
\left\Vert v_{1}-v_{0}\right\Vert _{W^{1,2}}\nonumber\\
&  \leq C\left(  M,\varepsilon,\delta\right)  \left\Vert v_{1}-v_{0}%
\right\Vert _{W^{1,2}}^{\frac{2}{5}}+C\left(  M,\varepsilon,\delta\right)
\left\Vert v_{1}-v_{0}\right\Vert _{W^{1,2}}.\label{continuite_T_1}%
\end{align}
Moreover, multiplying $\left(  \text{\ref{diferenta}}\right)  $
with $-\Delta\left(  \rho_{1}-\rho_{0}\right)  $ one gets%
\begin{align*}
&  \varepsilon\left\Vert \Delta\left(  \rho_{1}-\rho_{0}\right)  \right\Vert
_{L^{2}}^{2}+\delta\left\Vert \nabla\left(  \rho_{1}-\rho_{0}\right)
\right\Vert _{L^{2}}^{2}\\
&  \leq\left\Vert \Delta\left(  \rho_{1}-\rho_{0}\right)  \right\Vert _{L^{2}%
}\left\Vert \operatorname{div}\left(  \rho_{1}\omega_{\delta}\ast v_{1}%
-\rho_{0}\omega_{\delta}\ast v_{0}\right)  \right\Vert _{L^{2}},
\end{align*}
from which we deduce that%
\begin{equation}
\left\Vert \Delta\left(  \rho_{1}-\rho_{0}\right)  \right\Vert _{L^{2}%
}+\left\Vert \nabla\left(  \rho_{1}-\rho_{0}\right)  \right\Vert _{L^{2}}\leq
C\left(  M,\varepsilon,\delta\right)  \left(  \left\Vert v_{1}-v_{0}%
\right\Vert _{W^{1,2}}^{\frac{2}{5}}+\left\Vert v_{1}-v_{0}\right\Vert
_{W^{1,2}}\right)  \label{continuite_T_2}%
\end{equation}
Next, by taking the difference of the velocity equations we end up with%
\begin{align*}
\mathcal{A}\left(  S\left(  v_{1}\right)  -S\left(  v_{0}\right)  \right)
&  =\frac{\delta}{2}\left(  \rho_{0}v_{0}-\rho_{1}v_{1}\right)  -\frac{\delta
}{2}\left(  \int_{\mathbb{T}^{3}}\rho_{0}v_{0}-\int_{\mathbb{T}^{3}}\rho
_{1}v_{1}\right)  \\
&  +\operatorname{div}\left(  \rho_{1}\omega_{\delta}\ast v_{1}\otimes
v_{1}-\rho_{0}\omega_{\delta}\ast v_{0}\otimes v_{0}\right)  \\
&  +\nabla\omega_{\delta}\ast\left(  \rho_{1}^{\gamma}-\rho_{0}^{\gamma
}\right)  +\varepsilon\left(  \nabla v_{1}\nabla\rho_{1}-\nabla v_{0}%
\nabla\rho_{0}\right)  -\varepsilon\left(  \int_{\mathbb{T}^{3}}\nabla
v_{1}\nabla\rho_{1}-\int_{\mathbb{T}^{3}}\nabla v_{0}\nabla\rho_{0}\right)  ,
\end{align*}
from which we deduce that%
\begin{align*}
\left\Vert S\left(  v_{1}\right)  -S\left(  v_{0}\right)  \right\Vert
_{W^{1,2}} &  \lesssim\delta\left\Vert \rho_{0}v_{0}-\rho_{1}v_{1}\right\Vert
_{L^{\frac{6}{5}}}+\left\Vert \rho_{1}\omega_{\delta}\ast v_{1}\otimes
v_{1}-\rho_{0}\omega_{\delta}\ast v_{0}\otimes v_{0}\right\Vert _{L^{2}}\\
&  +\left\Vert \rho_{1}^{\gamma}-\rho_{0}^{\gamma}\right\Vert _{L^{2}%
}+2\varepsilon\left\Vert \nabla v_{1}\nabla\rho_{1}-\nabla v_{0}\nabla\rho
_{0}\right\Vert _{L^{\frac{6}{5}}}.
\end{align*}
Using $\left(  \text{\ref{continuite_T_1}}\right)  $, the first term is
treated as follows%
\begin{align}
\left\Vert \rho_{0}v_{0}-\rho_{1}v_{1}\right\Vert _{L^{\frac{6}{5}}} &
\leq\left\Vert v_{0}\right\Vert _{L^{3}}\left\Vert \left(  \rho_{1}-\rho
_{0}\right)  \right\Vert _{L^{2}}+\left\Vert \rho_{1}\right\Vert _{L^{\frac
{3}{2}}}\left\Vert v_{1}-v_{0}\right\Vert _{L^{6}}\nonumber\\
&  \leq C\left(  M,\varepsilon,\delta\right)  (\left\Vert v_{1}-v_{0}%
\right\Vert _{W^{1,2}}^{\frac{2}{5}}+\left\Vert v_{1}-v_{0}\right\Vert
_{W^{1,2}}).\label{Term0}%
\end{align}
The second term is estimated as follows%
\begin{align}
&  \left\Vert \rho_{1}\omega_{\delta}\ast v_{1}\otimes v_{1}-\rho_{0}%
\omega_{\delta}\ast v_{0}\otimes v_{0}\right\Vert _{L^{2}}\nonumber\\
&  \leq\left\Vert \omega_{\delta}\ast v_{1}\otimes v_{1}\right\Vert _{L^{3}%
}\left\Vert \rho_{1}-\rho_{0}\right\Vert _{L^{6}}+\left\Vert \rho
_{0}\right\Vert _{L^{6}}\left\Vert \omega_{\delta}\ast v_{1}\otimes
v_{1}-\omega_{\delta}\ast v_{0}\otimes v_{0}\right\Vert _{L^{3}}\nonumber\\
&  \leq\left\Vert v_{1}\right\Vert _{L^{6}}^{2}\left\Vert \rho_{1}-\rho
_{0}\right\Vert _{W^{1,2}}+\left\Vert \rho_{0}\right\Vert _{L^{6}}\left\{
\left\Vert v_{1}\right\Vert _{L^{6}}\left\Vert \omega_{\delta}\ast
v_{1}-\omega_{\delta}\ast v_{0}\right\Vert _{L^{6}}+\left\Vert \omega_{\delta
}\ast v_{0}\right\Vert _{L^{6}}\left\Vert v_{1}-v_{0}\right\Vert _{L^{6}%
}\right\}  \nonumber\\
&  \leq C\left(  M,\varepsilon,\delta\right)  (\left\Vert v_{1}-v_{0}%
\right\Vert _{W^{1,2}}^{\frac{2}{5}}+\left\Vert v_{1}-v_{0}\right\Vert
_{W^{1,2}}).  \label{Term1}%
\end{align}
The third term is treated using the Sobolev inequality along with $\left(
\text{\ref{Nouvelle_estimation_reg_rho}}\right)  $ and $\left(
\text{\ref{continuite_T_2}}\right)  $%
\begin{align}
\left\Vert \rho_{1}^{\gamma}-\rho_{0}^{\gamma}\right\Vert _{L^{2}} &
\leq\left(  \left\Vert \rho_{1}\right\Vert _{L^{\infty}}^{\gamma-1}+\left\Vert
\rho_{0}\right\Vert _{L^{\infty}}^{\gamma-1}\right)  \left\Vert \rho_{1}%
-\rho_{0}\right\Vert _{L^{2}}\nonumber\\
&  \leq C\left(  \left\Vert \rho_{1}\right\Vert _{W^{2,2}}^{\gamma
-1}+\left\Vert \rho_{0}\right\Vert _{W^{2,2}}^{\gamma-1}\right)  \left\Vert
\rho_{1}-\rho_{0}\right\Vert _{W^{1,2}}\nonumber\\
&  \leq C\left(  M,\varepsilon,\delta\right)  \left\Vert v_{1}-v_{0}%
\right\Vert _{W^{1,2}}^{\frac{2}{5}}+\left\Vert v_{1}-v_{0}\right\Vert
_{W^{1,2}}). \label{Term2}%
\end{align}
The forth term is treated with the help of relations $\left(
\text{\ref{Nouvelle_estimation_reg_rho}}\right)  $,$\left(
\text{\ref{continuite_T_2}}\right)  $ and $\left(  \text{\ref{continuite_T_1}%
}\right)  $
\begin{align}
&  \left\Vert \nabla v_{1}\nabla\rho_{1}-\nabla v_{0}\nabla\rho_{0}\right\Vert
_{L^{\frac{6}{5}}}\nonumber\\
&  \leq\left\Vert \nabla v_{1}\right\Vert _{L^{2}}\left\Vert \nabla(\rho
_{1}-\rho_{0})\right\Vert _{L^{3}}+\left\Vert \nabla\rho_{1}\right\Vert
_{L^{3}}\left\Vert \nabla v-\nabla v_{0}\right\Vert _{L^{2}}\nonumber\\
&  \leq C\left(  M,\varepsilon,\delta\right)  \left\Vert v_{1}-v_{0}%
\right\Vert _{W^{1,2}}^{\frac{2}{5}}+\left\Vert v_{1}-v_{0}\right\Vert
_{W^{1,2}}).\label{Term3}%
\end{align}
From $\left(  \text{\ref{Term0}}\right)  $, $\left(  \text{\ref{Term1}%
}\right)  $, $\left(  \text{\ref{Term2}}\right)  $, $\left(  \text{\ref{Term3}%
}\right)  $ we obtain that
\begin{equation}
\left\Vert S\left(  v_{1}\right)  -S\left(  v_{0}\right)  \right\Vert
_{W^{1,2}}\leq C\left(  M,\varepsilon,\delta\right)  \left\Vert v_{1}%
-v_{0}\right\Vert _{W^{1,2}}^{\frac{2}{5}}+\left\Vert v_{1}-v_{0}\right\Vert
_{W^{1,2}}).\label{continuite_S_1}%
\end{equation}
Of course, the above relation shows that $S$ is continuos. Moreover, using
$\left(  \text{\ref{compactness_in_u}}\right)  $ we get that the operator $S$
is compact.
\medskip\noindent2) \textit{The set }$\mathcal{P}$ \textit{defined in }$\left(  \text{\ref{set_P}}\right)  $\textit{ is bounded}. In the
following we prove that the set%
\[
\left\{  u\in\left(  W^{1,2}\left(  \mathbb{T}^{3}\right)  \right)
^{3}:u=\lambda S\left(  u\right)  \text{ for some }\lambda\in(0,1]\right\}
\]
is bounded. Thus, consider $\lambda\in(0,1]$ and $u\in\left(  W^{1,2}\left(
\mathbb{T}^{3}\right)  \right)  ^{3}$ such that $u=\lambda S\left(  u\right)
$. Obviously, one has%
\begin{equation}
\left\{
\begin{array}
[c]{l}%
-\varepsilon\Delta\rho+\delta\left(  \rho-M\right)  +\operatorname{div}\left(
\rho\omega_{\delta}\ast u\right)  =0,\\
\dfrac{\delta}{2}\left(  \rho u-%
{\displaystyle\int_{\mathbb{T}^{3}}}
\rho u\right)  -\dfrac{1}{\lambda}\mathcal{A}u+\operatorname{div}\left(
\rho\omega_{\delta}\ast u\otimes u\right)  +\nabla\left(  \omega_{\delta}%
\ast\rho^{\gamma}\right)  +\varepsilon\left(  \nabla u\nabla\rho-%
{\displaystyle\int_{\mathbb{T}^{3}}}
\nabla u\nabla\rho\right)  =\omega_{\delta}\ast g,\\%
{\displaystyle\int_{\mathbb{T}^{3}}}
\rho=M,\text{ }%
{\displaystyle\int_{\mathbb{T}^{3}}}
u=0.
\end{array}
\right.  \label{Leray-Schauder}%
\end{equation}
Observe that%
\begin{align}
&  \left\langle \operatorname{div}\left(  \rho\omega_{\delta}\ast u\otimes
u\right)  +\varepsilon\nabla u\nabla\rho,u\right\rangle \\
&  =\frac{1}{2}\operatorname{div}\left(  \rho\omega_{\delta}\ast u\left\vert
u\right\vert ^{2}\right)  +\operatorname{div}\left(  \rho\omega_{\delta}\ast
u\right)  \frac{\left\vert u\right\vert ^{2}}{2}+\dfrac{\varepsilon}%
{2}\left\langle \nabla\left\vert u\right\vert ^{2},\nabla\rho\right\rangle
\nonumber\\
&  =\frac{1}{2}\operatorname{div}\left(  \rho\omega_{\delta}\ast u\left\vert
u\right\vert ^{2}\right)  +\left(  \varepsilon\Delta\rho-\delta\left(
\rho-M\right)  \right)  \frac{\left\vert u\right\vert ^{2}}{2}+\dfrac
{\varepsilon}{2}\left\langle \nabla\left\vert u\right\vert ^{2},\nabla
\rho\right\rangle \nonumber\\
&  =\frac{1}{2}\operatorname{div}\left(  \rho\omega_{\delta}\ast u\left\vert
u\right\vert ^{2}\right)  +\dfrac{\varepsilon}{2}\operatorname{div}\left(
\left\vert u\right\vert ^{2}\nabla\rho\right)  -\delta\left(  \rho-M\right)
\frac{\left\vert u\right\vert ^{2}}{2}.\label{calcul_div(u_otimesu)}%
\end{align}
Next%
\[
\int u\nabla\left(  \omega_{\delta}\ast\rho^{\gamma}\right)  =-\int
\rho^{\gamma}\operatorname{div}\omega_{\delta}\ast u=\frac{4}{\gamma\left(
\gamma-1\right)  }\int_{\mathbb{T}^{3}}\left\vert \nabla\rho^{\frac{\gamma}%
{2}}\right\vert ^{2}+\gamma\delta\left(  \int_{\mathbb{T}^{3}}\rho^{\gamma
}-M\int_{\mathbb{T}^{3}}\rho^{\gamma-1}\right).
\]
Thus, we have that%
\begin{align}
&  -\frac{1}{\lambda}\int_{\mathbb{T}^{3}}\left\langle \mathcal{A}%
u,u\right\rangle +\frac{\delta M}{2}\int_{\mathbb{T}^{3}}\left\vert
u\right\vert ^{2}+\frac{4\varepsilon}{\gamma\left(  \gamma-1\right)  }%
\int_{\mathbb{T}^{3}}\left\vert \nabla\rho^{\frac{\gamma}{2}}\right\vert
^{2}+\gamma\delta\int_{\mathbb{T}^{3}}\rho^{\gamma}\nonumber\\
&  =\int_{\mathbb{T}^{3}}\omega_{\delta}\ast gu+\gamma\delta M\int
_{\mathbb{T}^{3}}\rho^{\gamma-1}.\label{bound_vitesse_ind_eps_delta}%
\end{align}
We use Young's inequality in order to obtain that
\begin{align*}
\gamma\delta M\int_{\mathbb{T}^{3}}\rho^{\gamma-1}+\int_{\mathbb{T}^{3}}%
\omega_{\delta}\ast gu &  \leq\gamma\delta M\left(  \int_{\mathbb{T}^{3}}%
\rho^{\gamma}\right)  ^{\frac{\gamma-1}{\gamma}}+\left\Vert g\right\Vert
_{L^{\frac{6}{5}}}\left\Vert u\right\Vert _{L^{6}}\\
&  \leq C\left(  M,\gamma\right)  (\left\Vert g\right\Vert _{L^{\frac{6}{5}}%
}^{2}+\delta)+\frac{1}{2\lambda}\int_{\mathbb{T}^{3}}\left\langle
\mathcal{A}u,u\right\rangle +\frac{\gamma\delta}{2}\int_{\mathbb{T}^{3}}%
\rho^{\gamma}.
\end{align*}
We obtain the existence of a constant $C\left(  M,\gamma\right)  $ depending
only on $M$ and $\gamma$ such that
\begin{align}
&  -\frac{1}{2\lambda}\int_{\mathbb{T}^{3}}\left\langle \mathcal{A}%
u,u\right\rangle +\frac{\delta M}{2}\int_{\mathbb{T}^{3}}\left\vert
u\right\vert ^{2}+\frac{4\varepsilon}{\gamma\left(  \gamma-1\right)  }%
\int_{\mathbb{T}^{3}}\left\vert \nabla\rho^{\frac{\gamma}{2}}\right\vert
^{2}+\frac{\gamma\delta}{2}\int_{\mathbb{T}^{3}}\rho^{\gamma}\nonumber\\
&  \leq C\left(  M,\gamma\right)  (\left\Vert g\right\Vert _{L^{\frac{6}{5}}%
}^{2}+\delta)\leq C\left(  M,\gamma\right)  (\left\Vert g\right\Vert
_{L^{\frac{6}{5}}}^{2}+1).\label{energy_estimate_eps_delta}%
\end{align}
The last estimate implies that $\mathcal{P}$ is a bounded set of $\left(
W^{1,2}\left(  \mathbb{T}^{3}\right)  \right)  ^{3}$. Having proved that the
operator $S$ verifies the hypothesis announced in Theorem \ref{Schauder_Leray}
we conclude that$\mathcal{\ }S$ admits a fixed point. This concludes the proof
of Proposition \ref{Point_fix}.\end{proof}

As an immediate consequence of Proposition \ref{Point_fix} we get the
following result

\begin{corollary}
\label{Corollary_point_fix}Consider $\left(  \varepsilon,\delta\right)
\in\left(  0,1\right)  ^{2}$. For all $M>0$ and $g\in(L^{\frac{3\left(
\gamma-1\right)  }{2\gamma-1}}\left(  \mathbb{T}^{3}\right)  )^{3}$ with $%
{\displaystyle\int_{\mathbb{T}^{3}}}
g=0$, there exists a solution $\left(  \rho^{\varepsilon,\delta}%
,u^{\varepsilon,\delta}\right)  \in W^{2,2}\left(  \mathbb{T}^{3}\right)
\times(W^{2,\frac{3}{2}}\left(  \mathbb{T}^{3}\right)  )^{3}$ of $\left(
\text{\ref{Approximare_eps_delta}}\right)  $ verifying the following
estimates:%
\begin{equation}
\left\{
\begin{array}
[c]{l}%
-\frac{1}{2}%
{\displaystyle\int_{\mathbb{T}^{3}}}
\left\langle \mathcal{A}u^{\varepsilon,\delta},u^{\varepsilon,\delta
}\right\rangle +\frac{4\varepsilon}{\gamma\left(  \gamma-1\right)  }%
{\displaystyle\int_{\mathbb{T}^{3}}}
\left\vert \nabla\left(  \rho^{\varepsilon,\delta}\right)  ^{\frac{\gamma}{2}%
}\right\vert ^{2}+\frac{\gamma\delta}{2}%
{\displaystyle\int_{\mathbb{T}^{3}}}
(\rho^{\varepsilon,\delta})^{\gamma}\leq C\left(  M,\gamma\right)  (\left\Vert
g\right\Vert _{L^{\frac{6}{5}}}^{2}+1),\\
\left\Vert \Delta\rho^{\varepsilon,\delta}\right\Vert _{L^{\frac{3}{2}}%
}+\left\Vert \nabla\rho^{\varepsilon,\delta}\right\Vert _{L^{2}}%
^{2}+\left\Vert \mathcal{A}u^{\varepsilon,\delta}\right\Vert _{L^{\frac{6}{5}%
}}\leq C(M,\varepsilon,\left\Vert g\right\Vert _{L^{\frac{6}{5}}}),
\end{array}
\right.  \label{estimari_eps_delta_prop}%
\end{equation}

\end{corollary}

\begin{proof}
The existence part of Corollary \ref{Corollary_point_fix} follows by observing
that a fixed point $u\in\left(  W^{1,2}\left(  \mathbb{T}^{3}\right)  \right)
^{3}$ of the operator $S$ defined by $\left(  \text{\ref{definitie_S_}%
}\right)  $ turns out to verify $\left(  \text{\ref{Approximare_eps_delta}%
}\right)  $. In order to finish the proof we must show that the pair $\left(
\rho^{\varepsilon,\delta},u^{\varepsilon,\delta}\right)  \in W^{2,2}\left(
\mathbb{T}^{3}\right)  \times\left(  W^{1,\frac{3}{2}}\left(  \mathbb{T}^{3}\right)
\right)  ^{3}$ constructed above verifies the announced estimates. We drop the
$\varepsilon,\delta$ upper scripts in order to render the computations easier
to follow. The first estimate of $\left(  \text{\ref{estimari_eps_delta_prop}}\right)  $ is nothing else but $\left(  \text{\ref{energy_estimate_eps_delta}%
}\right)  $ with $\lambda=1$. Of course, we will use it in order to prove the second inequality from $\left(  \text{\ref{estimari_eps_delta_prop}}\right)  $.
We begin with
\begin{align*}
\varepsilon\left\Vert \nabla\rho\right\Vert _{L^{2}}^{2}+\delta\left\Vert
\rho\right\Vert _{L^{2}}^{2} &  =\lambda\left(  \delta M^{2}+\int
_{\mathbb{T}^{3}}\omega_{\delta}\ast\operatorname{div}u\rho^{2}\right)  \\
&  \leq M^{2}+\left\Vert \operatorname{div}u\right\Vert _{L^{2}}\left\Vert
\rho\right\Vert _{L^{4}}^{2}\leq M^{2}+C\left(  1+\left\Vert g\right\Vert
_{L^{\frac{6}{5}}}\right)  \left\Vert \rho\right\Vert _{L^{1}}^{\frac{1}{5}%
}\left\Vert \rho\right\Vert _{L^{6}}^{\frac{9}{5}}%
\end{align*}
which by means of the Young inequality yields%
\begin{equation}
\varepsilon\left\Vert \nabla\rho\right\Vert _{L^{2}}^{2}+\delta\left\Vert
\rho\right\Vert _{L^{2}}^{2}\leq C(M,\varepsilon,\left\Vert g\right\Vert
_{L^{\frac{6}{5}}}).\label{bound_rho_ind_delta}%
\end{equation}
The estimate for the laplacian of $\rho$ is recovered using $\left(
\text{\ref{bound_rho_ind_delta}}\right)  $ and the Poincar\'{e} inequality:
\begin{align*}
\varepsilon\left\Vert \Delta\rho\right\Vert _{L^{\frac{3}{2}}} &  \leq
\delta\left\Vert \rho-M\right\Vert _{L^{\frac{3}{2}}}+\left\Vert
\rho\operatorname{div}\omega_{\delta}\ast u\right\Vert _{L^{\frac{3}{2}}%
}+\left\Vert \omega_{\delta}\ast u\nabla\rho\right\Vert _{L^{\frac{3}{2}}}\\
&  \leq\left\Vert \nabla\rho\right\Vert _{L^{2}}+\left\Vert \rho\right\Vert
_{L^{6}}\left\Vert \operatorname{div}\omega_{\delta}\ast u\right\Vert _{L^{2}%
}+\left\Vert \omega_{\delta}\ast u\right\Vert _{L^{6}}\left\Vert \nabla
\rho\right\Vert _{L^{2}}\\
&  \leq C\left(  M,\varepsilon,\left\Vert g\right\Vert _{L^{\frac{6}{5}}%
}\right)  .
\end{align*}
The last estimate along with Sobolev's inequality implies that%
\begin{equation}
\left\Vert \nabla\rho\right\Vert _{L^{3}}\leq C\left(  M,\varepsilon
,\left\Vert g\right\Vert _{L^{\frac{6}{5}}}\right)  .\label{rho_in_L3}%
\end{equation}
In the following lines, we analyze the terms appearing in the velocity
equation. Using $\left(  \text{\ref{bound_rho_ind_delta}}\right)  $ and
$\left(  \text{\ref{rho_in_L3}}\right)  $ we obtain%
\begin{equation}
\left\Vert \rho u\right\Vert _{L^{3}}+\left\Vert \nabla u\nabla\rho\right\Vert
_{L^{6/5}}\leq\left\Vert \rho\right\Vert _{L^{6}}\left\Vert u\right\Vert
_{L^{6}}+\left\Vert \nabla u\right\Vert _{L^{2}}\left\Vert \nabla
\rho\right\Vert _{L^{3}}\leq C\left(  M,\varepsilon,\left\Vert g\right\Vert
_{L^{\frac{6}{5}}}\right)  .\label{viteza_1}%
\end{equation}
Next, writing that%
\begin{align}
\left\Vert \nabla\left(  \omega_{\delta}\ast\rho^{\gamma}\right)  \right\Vert
_{L^{3/2}} &  =\left\Vert \omega_{\delta}\ast\left(  \rho^{\gamma/2}\nabla
\rho^{\gamma/2}\right)  \right\Vert _{L^{3/2}}\leq\left\Vert \rho^{\gamma
/2}\right\Vert _{L^{6}}\left\Vert \nabla\rho^{\gamma/2}\right\Vert _{L^{2}%
}\nonumber\\
&  \leq C\left\Vert \nabla\rho^{\gamma/2}\right\Vert _{L^{2}}^{2}\leq C\left(
M,\varepsilon,\left\Vert g\right\Vert _{L^{\frac{6}{5}}}\right)
.\label{viteza_2}%
\end{align}
Notice that we also have that%
\begin{align}
\left\Vert \operatorname{div}\left(  \rho\omega_{\delta}\ast u\otimes
u\right)  \right\Vert _{L^{\frac{6}{5}}} &  =\left\Vert \left(  \nabla
\rho\cdot\omega_{\delta}\ast u\right)  u\right\Vert _{L^{\frac{6}{5}}%
}+\left\Vert \rho\operatorname{div}(\omega_{\delta}\ast u)u\right\Vert
_{L^{\frac{6}{5}}}+\left\Vert \rho\omega_{\delta}\ast u\cdot\nabla
u\right\Vert _{L^{\frac{6}{5}}}\nonumber\\
&  \leq\left\Vert \nabla\rho\right\Vert _{L^{2}}\left\Vert u\right\Vert
_{L^{6}}^{2}+\left\Vert \rho\right\Vert _{L^{6}}\left\Vert \operatorname{div}%
u\right\Vert _{L^{2}}\left\Vert u\right\Vert _{L^{6}}+\left\Vert
\rho\right\Vert _{L^{6}}\left\Vert u\right\Vert _{L^{6}}\left\Vert \nabla
u\right\Vert _{L^{2}}\nonumber\\
&  \leq C\left(  M,\varepsilon,\left\Vert g\right\Vert _{L^{\frac{6}{5}}%
}\right)  \label{viteza_3}%
\end{align}
From $\left(  \text{\ref{viteza_1}}\right)  $, $\left(  \text{\ref{viteza_2}%
}\right)  $ and $\left(  \text{\ref{viteza_3}}\right)  $ it follows that
$\mathcal{A}u\in\left(  L^{6/5}\left(  \mathbb{T}^{3}\right)  \right)  ^{3}$
with
\[
\left\Vert \mathcal{A}u\right\Vert _{L^{\frac{6}{5}}}\leq C\left(
M,\varepsilon,\left\Vert g\right\Vert _{L^{\frac{6}{5}}}\right)  .
\]
This concludes the proof of Corollary \ref{Corollary_point_fix}.
\end{proof}

\section{Proof of Theorem \ref{Main1} \label{TE}}

To obtain Theorem \ref{Main1} from the approximate system
\eqref{Approximare_eps_delta}, it remains to pass to the limit first with
respect to $\delta$ and secondly with respect to $\varepsilon$. As usually, in
order to use the nonlinear weak stability obtained in a previous section, one
important step will be to obtain estimates uniformly with respect to
$\varepsilon$.

\subsection{The approximate system in the limit $\delta\rightarrow
0$\label{delta_goes_to_0}}

Owing to the Corollary \ref{Point_fix} we see that for any $\varepsilon
,\delta\in\left(  0,1\right)  $ we may consider $\left(  \rho^{\varepsilon
,\delta},u^{\varepsilon,\delta}\right)  \in W^{2,\frac{3}{2}}\left(
\mathbb{T}^{3}\right)  \times(W^{2,\frac{3}{2}}\left(  \mathbb{T}^{3}\right)
)^{3}$ solution of $\left(  \text{\ref{Approximare_eps_delta}}\right)  $ which
verifies, uniformly in $\delta$ the estimates announced in $\left(
\text{\ref{estimari_eps_delta_prop}}\right)  $. By virtue of the
Rellich--Kondrachov theorem, these estimates are sufficient in order to pass
to the limit when $\delta$ tends to $0$ and obtain a solution of the limit
system verifying the first estimate in $\left(
\text{\ref{estimari_eps_delta_prop}}\right)  $. We skip the details. More
precisely, we obtain the following:

\begin{proposition}
\bigskip\label{existence_system_eps}Consider $\varepsilon\in\left(
0,1\right)  $. For all $M>0$ and $g\in(L^{\frac{3\left(  \gamma-1\right)
}{2\gamma-1}}\left(  \mathbb{T}^{3}\right)  )^{3}$ with $%
{\displaystyle\int_{\mathbb{T}^{3}}}
g=0$, there exists $\left(  \rho^{\varepsilon},u^{\varepsilon}\right)  \in
W^{2,\frac{3}{2}}\left(  \mathbb{T}^{3}\right)  \times(W^{2,\frac{6}{5}%
}\left(  \mathbb{T}^{3}\right)  )^{3}$ verifying
\begin{equation}
\left\{
\begin{array}
[c]{l}%
-\varepsilon\Delta\rho^{\varepsilon}+\operatorname{div}\left(  \rho
^{\varepsilon}u^{\varepsilon}\right)  =0,\\
\operatorname{div}\left(  \rho u^{\varepsilon}\otimes u^{\varepsilon}\right)
-\mathcal{A}u^{\varepsilon}+\nabla(\rho^{\varepsilon})^{\gamma}+\varepsilon
(\nabla u^{\varepsilon}\nabla\rho^{\varepsilon}-%
{\displaystyle\int_{\mathbb{T}^{3}}}
\nabla u^{\varepsilon}\nabla\rho^{\varepsilon})=g,\\
\rho^{\varepsilon}\geq0,\text{ }%
{\displaystyle\int_{\mathbb{T}^{3}}}
\rho^{\varepsilon}=M,\text{ }%
{\displaystyle\int_{\mathbb{T}^{3}}}
u^{\varepsilon}=0.
\end{array}
\right.  \label{system_eps_Prop}%
\end{equation}
along with the estimates
\begin{equation}
\left\{
\begin{array}
[c]{l}%
-%
{\displaystyle\int_{\mathbb{T}^{3}}}
\left\langle \mathcal{A}u^{\varepsilon},u^{\varepsilon}\right\rangle
+\varepsilon\tfrac{4}{\gamma\left(  \gamma-1\right)  }%
{\displaystyle\int_{\mathbb{T}^{3}}}
\left\vert \nabla(\rho^{\varepsilon})^{\frac{\gamma}{2}}\right\vert ^{2}\leq
C_{0}\left(  1+\left\Vert g\right\Vert _{L^{\frac{6}{5}}}^{2}\right)  ,\\
\left\Vert \rho^{\varepsilon}\right\Vert _{L^{3\left(  \gamma-1\right)  }%
}+\left\Vert \nabla u^{\varepsilon}\right\Vert _{L^{\frac{3\left(
\gamma-1\right)  }{\gamma}}}\leq C,\\
\varepsilon\left\Vert \nabla u^{\varepsilon}\nabla\rho^{\varepsilon
}\right\Vert _{L^{\frac{3(\gamma-1)}{2\gamma-1}}}\leq\varepsilon^{\frac
{\theta}{2}}C\text{ for some }\theta\in(0,1),
\end{array}
\right.  \label{energy_estimate}%
\end{equation}
Where $C_{0}$ and $C=C(\theta,\mu,\lambda,\gamma,\left\Vert g\right\Vert
_{L^{\frac{3(\gamma-1)}{2\gamma-1}}},\left\Vert \eta\right\Vert _{L^{\frac
{6\left(  \gamma-1\right)  }{4\gamma-3}}},\left\Vert \xi\right\Vert
_{L^{\frac{6\left(  \gamma-1\right)  }{4\gamma-3}}},M)$ are positive constants
independent of $\varepsilon$.
\end{proposition}

We fill focus instead on proving the second and third estimates announced in
$\left(  \text{\ref{energy_estimate}}\right)  $ which say that it is possible
to recover estimates for the density that are independent of $\varepsilon$
along with better integrability properties for the velocity $u$. This is the
objective of the next section.

\subsubsection{Estimates for the density and improved estimates for the
velocity\label{Estimates_density_velocity}}

We will drop the $\varepsilon$ superscript in order to ease the reading of the
computation that follow. Thus, consider a pair $\left(  \rho,u\right)  \in
W^{2,\frac{3}{2}}\left(  \mathbb{T}^{3}\right)  \times(W^{2,\frac{6}{5}%
}\left(  \mathbb{T}^{3}\right)  )^{3}$ solution of $\left(
\text{\ref{system_eps_Prop}}\right)  $ verifying the first estimate from
$\left(  \text{\ref{energy_estimate}}\right)  $. Apply the divergence
$\operatorname{div}$ operator in the momentum equation such as to obtain%
\begin{equation}
-\left\{  \left(  \mu\Delta_{\theta}+\left(  \mu+\lambda\right)
\Delta\right)  \operatorname{div}u+\Delta\left(  (\eta+\xi)\ast
\operatorname{div}u\right)  \right\}  +\Delta\rho^{\gamma}=\operatorname{div}%
g-\operatorname{div}\operatorname{div}\left(  \rho u\otimes u\right)
-\varepsilon\operatorname{div}\left(  \nabla u\nabla\rho\right)
.\label{div_in_momentum}%
\end{equation}
from which we obtain that%
\begin{align*}
\rho^{\gamma} &  =\int_{\mathbb{T}^{3}}\rho^{\gamma}+\left(  2\mu
+\lambda\right)  \operatorname{div}u+\left(  Id-\left(  2\mu+\lambda\right)
\left(  \mu\Delta_{\theta}+\left(  \mu+\lambda\right)  \Delta\right)
^{-1}\Delta\right)  \left(  \rho^{\gamma}-\int_{\mathbb{T}^{3}}\rho^{\gamma
}\right)  \\
&  +\left(  2\mu+\lambda\right)  \left(  \mu\Delta_{\theta}+\left(
\mu+\lambda\right)  \Delta\right)  ^{-1}\left\{  \Delta(\eta+\xi
)\ast\operatorname{div}u+\operatorname{div}g-\operatorname{div}%
\operatorname{div}\left(  \rho u\otimes u\right)  -\varepsilon
\operatorname{div}\left(  \nabla u\nabla\rho\right)  \right\}  \\
&  \overset{not.}{=}\sum_{i=1}^{7}T_{i}.
\end{align*}
In the following we will search for an $\alpha>0$ such that all $i\in
\overline{1,7}$
\[
\int\rho^{\alpha}T_{i}\leq C\left(  \left\Vert g\right\Vert ,M\right)
+\beta\int_{\mathbb{T}^{3}}\rho^{\alpha+\gamma},
\]
with a sufficiently small $\beta$ from which we will obtain an estimate of the
form
\[
\int_{\mathbb{T}^{3}}\rho^{\alpha+\gamma}\leq C\left(  \left\Vert g\right\Vert
,M\right)  .
\]

\smallskip

\noindent\textit{First term $T_{1}$} We simply write that%
\begin{align}
\int_{\mathbb{T}^{3}}\rho^{\alpha}\int_{\mathbb{T}^{3}}\rho^{\gamma} &
\leq\frac{\alpha}{\alpha+\gamma}\left(  \int_{\mathbb{T}^{3}}\rho^{\alpha
}\right)  ^{\frac{\alpha+\gamma}{\alpha}}+\frac{\gamma}{\alpha+\gamma}\left(
\int_{\mathbb{T}^{3}}\rho^{\gamma}\right)  ^{\frac{\alpha+\gamma}{\gamma}%
}\nonumber\\
&  \leq\frac{\alpha}{\alpha+\gamma}\left(  \int_{\mathbb{T}^{3}}\rho\right)
^{(1-\theta_{1})\left(  \alpha+\gamma\right)  }\left(  \int_{\mathbb{T}^{3}%
}\rho^{\alpha+\gamma}\right)  ^{\theta_{1}\left(  \alpha+\gamma\right)
}\nonumber\\
&  +\frac{\gamma}{\alpha+\gamma}\left(  \int_{\mathbb{T}^{3}}\rho\right)
^{(1-\theta_{2})\left(  \alpha+\gamma\right)  }\left(  \int_{\mathbb{T}^{3}%
}\rho^{\alpha+\gamma}\right)  ^{\theta_{2}\left(  \alpha+\gamma\right)
}\nonumber\\
&  \leq C\left(  \alpha,\gamma,M,\beta\right)  +\beta\int_{\mathbb{T}^{3}}%
\rho^{\alpha+\gamma},\label{apriori_first_term}%
\end{align}
for any $\beta>0$ and some $\theta_{1},\theta_{2}\in(0,1)$.

\smallskip\noindent\textit{Second Term $T_{2}$.} Using the equation of $\rho$
we see that
\[
-\varepsilon\Delta\rho^{\alpha}+\varepsilon\frac{4}{\alpha}\left\vert
\nabla\rho^{\frac{\alpha}{2}}\right\vert ^{2}+\operatorname{div}(\rho^{\alpha
}u)+\left(  \alpha-1\right)  \rho^{\alpha}\operatorname{div}u=0,
\]
and consequently $T_{2}$ is a negative term:%
\begin{equation}
\int_{\mathbb{T}^{3}}\rho^{\alpha}T_{2}=\int_{\mathbb{T}^{3}}\rho^{\alpha
}\operatorname{div}u=-\varepsilon\dfrac{4}{\alpha\left(  \alpha-1\right)
}\int_{\mathbb{T}^{3}}\left\vert \nabla\rho^{\frac{\alpha}{2}}\right\vert
^{2}. \label{apriori_second_term}%
\end{equation}

\smallskip

\noindent\textit{Third term $T_{3}$.} The third term is more delicate to treat
because it is of the same order as $\rho^{\alpha+\gamma}$ such that we need
the smallness assumption $\left(  \text{\ref{smallness}}\right)  $. Again
using the mass equation we have that%
\begin{align}
&  \int_{\mathbb{T}^{3}}\rho^{\alpha}T_{3}\nonumber\\
&  \leq\left\Vert \rho^{\alpha}\right\Vert _{L^{\frac{\alpha+\gamma}{\alpha}}%
}\left\Vert \left(  Id-\left(  2\mu+\lambda\right)  \left(  \mu\Delta_{\theta
}+\left(  \mu+\lambda\right)  \Delta\right)  ^{-1}\Delta\right)  \left(
\rho^{\gamma}-\int_{\mathbb{T}^{3}}\rho^{\gamma}\right)  \right\Vert
_{L^{\frac{\alpha+\gamma}{\gamma}}}\nonumber\\
&  \leq C(1+|\theta|)\mu\left\vert \theta\right\vert \frac{2\lambda+\mu
}{\left(  \mu+\lambda\right)  ^{2}}\left\Vert \rho^{\alpha}\right\Vert
_{L\frac{^{\alpha+\gamma}}{\alpha}}\left\Vert \rho^{\gamma}\right\Vert
_{L^{\frac{\alpha+\gamma}{\gamma}}}=C(1+|\theta|)\mu\left\vert \theta
\right\vert \frac{2\lambda+\mu}{\left(  \mu+\lambda\right)  ^{2}}%
\int_{\mathbb{T}^{3}}\rho^{\alpha+\gamma},\label{apriori_third_term_1}%
\end{align}
where we have used that the norm of the operator
\[
Id-\left(  2\mu+\lambda\right)  \left(  \mu\Delta_{\theta}+\left(  \mu
+\lambda\right)  \Delta\right)  ^{-1}\Delta
\]
is controlled by
\[
C(1+|\theta|)\mu\left\vert \theta\right\vert \dfrac{2\lambda+\mu}{(\mu
+\lambda)^{2}},
\]
see Appendix after Theorem \ref{multipliers}. Consequently if this quantity is
sufficiently small, we will be able to close the estimates absorbing this term
by the left-hand side. This is satisfied for instance if $|\theta|$ is small
enough or the bulk viscosity large enough.

\smallskip\noindent\textit{Fourth term $T_{4}$.} The fourth term is treated as
follows%
\begin{align}
T_{4}  &  \leq C\left(  \theta,\mu,\lambda\right)  \left\Vert \rho\right\Vert
_{L^{\alpha+\gamma}}^{\alpha}\left\Vert \eta+\xi\right\Vert _{L^{\frac
{2\left(  \alpha+\gamma\right)  }{2\gamma+\alpha}}}\left\Vert
\operatorname{div}u\right\Vert _{L^{2}}\nonumber\\
&  \leq C\left(  \theta,\mu,\lambda\right)  \left\Vert \rho\right\Vert
_{L^{\alpha+\gamma}}^{\alpha}\left\Vert \eta+\xi\right\Vert _{L^{\frac
{2\left(  \alpha+\gamma\right)  }{2\gamma+\alpha}}}\left\Vert g\right\Vert
_{L^{\frac{6}{5}}}. \label{apriori_forth_term}%
\end{align}

\smallskip\noindent\textit{Fifth term $T_{5}$.} The fifth term is treated as
follows%
\begin{equation}
T_{5}\leq C\left(  \theta,\mu,\lambda\right)  \left\Vert \rho\right\Vert
_{L^{\alpha+\gamma}}^{\alpha}\left\Vert g\right\Vert _{L^{\frac{3\left(
\alpha+\gamma\right)  }{4\gamma+\alpha}}}, \label{apriori_fifth}%
\end{equation}
provided that%
\[
\frac{3\left(  \alpha+\gamma\right)  }{4\gamma+\alpha}>1\text{ which yields
}2\alpha>\gamma.
\]

\smallskip\noindent\textit{Sixth term $T_{6}$.} The sixth term is treated as
follows
\begin{align}
T_{6}  &  \leq C\left(  \theta,\mu,\lambda\right)  \left\Vert \rho\right\Vert
_{L^{\alpha+\gamma}}^{\alpha}\left\Vert \rho u\otimes u\right\Vert
_{L^{\frac{\alpha+\gamma}{\gamma}}}\leq C\left(  \theta,\mu,\lambda\right)
\left\Vert \rho\right\Vert _{L^{\alpha+\gamma}}^{\alpha}\left\Vert
u\right\Vert _{L^{6}}^{2}\left\Vert \rho\right\Vert _{L^{\frac{3\left(
\alpha+\gamma\right)  }{2\gamma-\alpha}}}\nonumber\\
&  \leq C\left(  \theta,\mu,\lambda\right)  \left\Vert g\right\Vert
_{L^{\frac{6}{5}}}^{2}\left\Vert \rho\right\Vert _{L^{\alpha+\gamma}%
}^{1+\alpha}. \label{apriori_sixth_term}%
\end{align}
Of course in order to pass to the second line of $\left(
\text{\ref{apriori_sixth_term}}\right)  $ we need to have%
\[
\frac{3\left(  \alpha+\gamma\right)  }{2\gamma-\alpha}\leq\alpha+\gamma\text{
which yields }\alpha\leq2\gamma-3\text{.}%
\]
This is the point where we see that a rather large adiabatic coefficient
$\gamma$ is needed in order to recover that the pressure is a bit better than
$L^{2}$.

\smallskip\noindent\textit{Seventh term $T_{7}$.} The seventh term is treated
as follows. First we write that%
\[
-\varepsilon\operatorname{div}\left(  \nabla u\nabla\rho\right)
=\operatorname{div}\left(  \nabla u\Delta^{-1}\nabla\operatorname{div}\left(
\rho u\right)  \right)  .
\]
Next, using the Sobolev inequality we get that%
\begin{align}
T_{7}  &  \leq\left\Vert \rho\right\Vert _{L^{\alpha+\gamma}}^{\alpha
}\left\Vert \nabla u\Delta^{-1}\nabla\operatorname{div}\left(  \rho u\right)
\right\Vert _{L^{\frac{3\left(  \alpha+\gamma\right)  }{4\gamma+\alpha}}%
}\nonumber\\
&  \leq\left\Vert \rho\right\Vert _{L^{\alpha+\gamma}}^{\alpha}\left\Vert
\nabla u\right\Vert _{L^{2}}\left\Vert \Delta^{-1}\nabla\operatorname{div}%
\left(  \rho u\right)  \right\Vert _{L^{\frac{6\left(  \alpha+\gamma\right)
}{5\gamma-\alpha}}}\leq\left\Vert \rho\right\Vert _{L^{\alpha+\gamma}}%
^{\alpha}\left\Vert \nabla u\right\Vert _{L^{2}}\left\Vert \rho u\right\Vert
_{L^{\frac{6\left(  \alpha+\gamma\right)  }{5\gamma-\alpha}}}\nonumber\\
&  \leq\left\Vert \rho\right\Vert _{L^{\alpha+\gamma}}^{\alpha}\left\Vert
\nabla u\right\Vert _{L^{2}}\left\Vert u\right\Vert _{L^{6}}\left\Vert
\rho\right\Vert _{L^{\frac{3\left(  \alpha+\gamma\right)  }{2\gamma-\alpha}}%
}\nonumber\\
&  \leq\left\Vert \rho\right\Vert _{L^{\alpha+\gamma}}^{1+\alpha}\left\Vert
\nabla u\right\Vert _{L^{2}}^{2}. \label{apriori_seventh_term}%
\end{align}

\smallskip\noindent\textit{Conclusion.} Finally, choosing $\alpha=2\gamma-3$
and putting together all the above estimates concerning $T_{i}$ for
$i=1,\cdots,7$ we get that
\begin{equation}
\left\Vert \rho\right\Vert _{L^{3\left(  \gamma-1\right)  }}\leq C\left(
\theta,\mu,\lambda,\gamma,\left\Vert g\right\Vert _{L^{\frac{3(\gamma
-1)}{2\gamma-1}}},\left\Vert \eta\right\Vert _{L^{\frac{6\left(
\gamma-1\right)  }{4\gamma-3}}},\left\Vert \xi\right\Vert _{L^{\frac{6\left(
\gamma-1\right)  }{4\gamma-3}}},M\right)  .\label{bound_on_density}%
\end{equation}
Of course, going back to the identity $\left(  \text{\ref{div_in_momentum}%
}\right)  $ and using the uniform bound on $\rho^{\gamma}$ in $L^{\frac
{3\left(  \gamma-1\right)  }{\gamma}}$ and proceeding as we did in estimate
$\left(  \text{\ref{apriori_seventh_term}}\right)  $ we can recover that%
\begin{equation}
\left\Vert \operatorname{div}u\right\Vert _{L^{\frac{3(\gamma-1)}{\gamma}}%
}+\left\Vert \rho\right\Vert _{L^{3\left(  \gamma-1\right)  }}\leq C\left(
\theta,\mu,\lambda,\gamma,\left\Vert g\right\Vert _{L^{\frac{3(\gamma
-1)}{2\gamma-1}}},\left\Vert \eta\right\Vert _{L^{\frac{6\left(
\gamma-1\right)  }{4\gamma-3}}},\left\Vert \xi\right\Vert _{L^{\frac{6\left(
\gamma-1\right)  }{4\gamma-3}}},M\right)  .\label{bound_on_density_and_div}%
\end{equation}
The last estimate can be used to get \textit{extra-integrability for the
velocity field} with respect to the basic energy estimate. This is achieved by
observing that%
\begin{align*}
\mu\nabla u &  =\Delta_{\theta}^{-1}\nabla\operatorname{div}\left(  \rho
u\otimes u\right)  +\Delta_{\theta}^{-1}\nabla^{2}\rho^{\gamma}--\left(
\mu+\lambda\right)  \Delta_{\theta}^{-1}\nabla^{2}\operatorname{div}%
u-\Delta_{\theta}^{-1}\nabla g\\
&  -\Delta_{\theta}^{-1}\Delta\left(  \eta\ast\nabla u\right)  -\Delta
_{\theta}^{-1}\nabla^{2}\left(  \xi\ast\operatorname{div}u\right)
+\varepsilon\Delta_{\theta}^{-1}\nabla(\nabla u\nabla\rho)
\end{align*}
such that we obtain
\begin{equation}
\left\Vert \nabla u\right\Vert _{L^{\frac{3(\gamma-1)}{\gamma}}}\leq C\left(
\theta,\mu,\lambda,\gamma,\left\Vert g\right\Vert _{L^{\frac{3(\gamma
-1)}{2\gamma-1}}},\left\Vert \eta\right\Vert _{L^{\frac{6\left(
\gamma-1\right)  }{4\gamma-3}}},\left\Vert \xi\right\Vert _{L^{\frac{6\left(
\gamma-1\right)  }{4\gamma-3}}},M\right)  .\label{improved_bound_on_velocity}%
\end{equation}

\smallskip\noindent\textit{Estimates for the gradient of the density.}
Finally, we aim at recovering some improved estimates for the gradient of
$\rho$. In order to do that, we write in a first time that%
\begin{equation}
\varepsilon\left\Vert \nabla\rho\right\Vert _{L^{2}}^{2}=\int\rho
^{2}\operatorname{div}u\leq\left\Vert \rho\right\Vert _{L^{4}}^{2}\left\Vert
\operatorname{div}u\right\Vert _{L^{2}}\leq C\left(  \theta,\mu,\lambda
,\gamma,\left\Vert g\right\Vert _{L^{\frac{3(\gamma-1)}{2\gamma-1}}%
},\left\Vert \eta\right\Vert _{L^{\frac{6\left(  \gamma-1\right)  }{4\gamma
-3}}},\left\Vert \xi\right\Vert _{L^{\frac{6\left(  \gamma-1\right)  }%
{4\gamma-3}}},M\right)  .\label{gradient_rho1}%
\end{equation}
Also, using%
\[
\varepsilon\nabla\rho=\Delta^{-1}\nabla\operatorname{div}\left(  \rho
u\right)  ,
\]
we get
\begin{equation}
\varepsilon\left\Vert \nabla\rho\right\Vert _{L^{\frac{3\left(  \gamma
-1\right)  }{2}}}\leq\left\Vert \rho\right\Vert _{L^{3\left(  \gamma-1\right)
}}\left\Vert u\right\Vert _{L^{3\left(  \gamma-1\right)  }}\leq C\left(
\theta,\mu,\lambda,\gamma,\left\Vert g\right\Vert _{L^{\frac{3(\gamma
-1)}{2\gamma-1}}},\left\Vert \eta\right\Vert _{L^{\frac{6\left(
\gamma-1\right)  }{4\gamma-3}}},\left\Vert \xi\right\Vert _{L^{\frac{6\left(
\gamma-1\right)  }{4\gamma-3}}},M\right)  .\label{gradient_rho2}%
\end{equation}
Using $\left(  \text{\ref{gradient_rho1}}\right)  $ and $\left(
\text{\ref{gradient_rho2}}\right)  $ we obtain
\[
\varepsilon\left\Vert \nabla\rho\right\Vert _{L^{3}}\leq\varepsilon\left\Vert
\nabla\rho\right\Vert _{L^{2}}^{\theta}\left\Vert \rho\right\Vert
_{L^{\frac{3\left(  \gamma-1\right)  }{2}}}^{1-\theta}\leq\varepsilon
^{\frac{\theta}{2}}C\left(  \theta,\mu,\lambda,\gamma,\left\Vert g\right\Vert
_{L^{\frac{3(\gamma-1)}{2\gamma-1}}},\left\Vert \eta\right\Vert _{L^{\frac
{6\left(  \gamma-1\right)  }{4\gamma-3}}},\left\Vert \xi\right\Vert
_{L^{\frac{6\left(  \gamma-1\right)  }{4\gamma-3}}},M\right)  ,
\]
where $\theta\in(0,1)$ is given by%
\[
\frac{1}{3}=\frac{\theta}{2}+\frac{2\left(  1-\theta\right)  }{3\left(
\gamma-1\right)  }.
\]
Moreover,
\begin{align}
\varepsilon\left\Vert \nabla u\nabla\rho\right\Vert _{L^{\frac{3(\gamma
-1)}{2\gamma-1}}} &  \leq\varepsilon\left\Vert \nabla u\right\Vert
_{L^{\frac{3\left(  \gamma-1\right)  }{\gamma}}}\left\Vert \nabla
\rho\right\Vert _{L^{3}}\nonumber\\
&  \leq\varepsilon^{\frac{\theta}{2}}C\left(  \theta,\mu,\lambda
,\gamma,\left\Vert g\right\Vert _{L^{\frac{3(\gamma-1)}{2\gamma-1}}%
},\left\Vert \eta\right\Vert _{L^{\frac{6\left(  \gamma-1\right)  }{4\gamma
-3}}},\left\Vert \xi\right\Vert _{L^{\frac{6\left(  \gamma-1\right)  }%
{4\gamma-3}}},M\right)  \label{the_extra_term_momentum_goes_to_0}%
\end{align}

\subsection{The limit passage $\varepsilon\to0$.}

The Proof of Theorem \ref{Main1} is based on the existence of solutions for
the regularized system $\left(  \text{\ref{system_eps_Prop}}\right)  $ and on
an adoption of the proof of the stability result Theorem \ref{Main1}. Owing to
Proposition \ref{existence_system_eps}, let us consider a sequence $\left(
\rho^{\varepsilon},u^{\varepsilon}\right)  _{\varepsilon>0}\subset
W^{2,\frac{3}{2}}\left(  \mathbb{T}^{3}\right)  \times(W^{2,\frac{6}{5}%
}\left(  \mathbb{T}^{3}\right)  )^{3}$ verifying $\left(
\text{\ref{system_eps_Prop}}\right)  $ and uniformly in $\varepsilon$ the
estimate $\left(  \text{\ref{energy_estimate}}\right)  $. Using the theory of
Sobolev spaces and the Rellich-Kondrachov theorem, we get the existence of
functions $\left(  \rho,u,\overline{\rho^{\gamma}},\overline{\mathcal{C(}%
u,u)}\right)  $ verifying%
\begin{equation}
\left\{
\begin{array}
[c]{l}%
\rho^{\varepsilon}\rightharpoonup\rho\text{ weakly\ in }L^{3(\gamma-1)}\left(
\mathbb{T}^{3}\right)  ,\\
(\rho^{\varepsilon})^{\gamma}\rightharpoonup\overline{\rho^{\gamma}%
}\text{\ weakly in }L^{\frac{3\left(  \gamma-1\right)  }{\gamma}}\left(
\mathbb{T}^{3}\right)  ,\\
\nabla u^{\varepsilon}\rightharpoonup\nabla u\text{ weakly in }L^{\frac
{3\left(  \gamma-1\right)  }{\gamma}}\left(  \mathbb{T}^{3}\right)  ,\\
\mathcal{C(}u^{\varepsilon},u^{\varepsilon})\rightharpoonup\overline
{\mathcal{C(}u,u)}\text{ weakly in }L^{\frac{3\left(  \gamma-1\right)
}{2\gamma}}\left(  \mathbb{T}^{3}\right)  ,\\
u^{\varepsilon}\rightarrow u\text{ strongly in }L^{q}\left(  \mathbb{T}%
^{3}\right)  \text{ for any }1\leq q<3\left(  \gamma-1\right)  .
\end{array}
\right.  \label{convergence_properties_2}%
\end{equation}
We recall that $\mathcal{C}$ is defined in relation $\left(
\text{\ref{definition_B_et_C}}\right)  $. We deduce that%
\begin{equation}
\left\{
\begin{array}
[c]{l}%
\operatorname{div}\left(  \rho u\right)  =0,\\
\operatorname{div}\left(  \rho u\otimes u\right)  -\mathcal{A}u+\nabla
\overline{\rho^{\gamma}}=g,\\%
{\displaystyle\int_{\mathbb{T}^{3}}}
\rho\left(  x\right)  dx=M,\text{ }%
{\displaystyle\int_{\mathbb{T}^{3}}}
u\left(  x\right)  dx=0,\text{ }\rho\geq0.
\end{array}
\right.  \label{limit_before_identification_2}%
\end{equation}
In order to identify $\overline{\rho^{\gamma}}$ with $\rho^{\gamma}$ we may
proceed exactly as we did in Section \ref{NWS} the only difference being that
we have%
\[
\frac{1}{\gamma-1}\operatorname{div}\left(  u\left(  \overline{\rho^{\gamma}%
}-\rho^{\gamma}\right)  \right)  +\left(  \overline{\rho^{\gamma}}%
-\rho^{\gamma}\right)  \operatorname{div}u+\overline{\mathcal{C}\left(
u,u\right)  }-C\left(  u,u\right)  \leq0,
\]
instead of $\left(  \text{\ref{Identity}}\right)  $. Indeed, the negative sign
comes from the fact that when we write the energy equations
\begin{align*}
&  -\frac{\varepsilon}{\gamma-1}\Delta(\rho^{\varepsilon})^{\gamma}%
+\frac{4\varepsilon}{\gamma\left(  \gamma-1\right)  }\left\vert \nabla
(\rho^{\varepsilon})^{\frac{\gamma}{2}}\right\vert ^{2}+\frac{\gamma}%
{\gamma-1}\operatorname{div}\left(  (\rho^{\varepsilon})^{\gamma
}u^{\varepsilon}\right)  \\
&  =u^{\varepsilon}\nabla(\rho^{\varepsilon})^{\gamma}\\
&  =-u^{\varepsilon}\left\{  \operatorname{div}\left(  \rho u^{\varepsilon
}\otimes u^{\varepsilon}\right)  +\varepsilon\nabla u^{\varepsilon}\nabla
\rho^{\varepsilon}\right\}  +\left\langle \mathcal{A}u^{\varepsilon
},u^{\varepsilon}\right\rangle +gu^{\varepsilon}+\varepsilon u^{\varepsilon
}\int_{\mathbb{T}^{3}}\nabla u^{\varepsilon}\nabla\rho^{\varepsilon}\\
&  =-\frac{1}{2}\operatorname{div}\left(  u^{\varepsilon}\rho^{\varepsilon
}\left\vert u^{\varepsilon}\right\vert ^{2}\right)  -\frac{\varepsilon}%
{2}\operatorname{div}\left(  u^{\varepsilon}\otimes\nabla\rho^{\varepsilon
}\right)  +\mathcal{B}\left(  u^{\varepsilon},u^{\varepsilon}\right)
-\mathcal{C}\left(  u^{\varepsilon},u^{\varepsilon}\right)  +gu^{\varepsilon
}+\varepsilon u^{\varepsilon}\int_{\mathbb{T}^{3}}\nabla u^{\varepsilon}%
\nabla\rho^{\varepsilon}.
\end{align*}
Thus, using $\left(  \text{\ref{convergence_property}}\right)  $ and $\left(
\text{\ref{the_extra_term_momentum_goes_to_0}}\right)  $ we get that%
\begin{equation}
\frac{\gamma}{\gamma-1}\operatorname{div}\left(  \overline{\rho^{\gamma}%
}u\right)  =-\frac{1}{2}\operatorname{div}\left(  u\rho\left\vert u\right\vert
^{2}\right)  +\mathcal{B}\left(  u,u\right)  -\overline{\mathcal{C}\left(
u,u\right)  }+gu-\Xi,\label{diference1}%
\end{equation}
where $\Xi$ is the limiting positive measure
\begin{equation}
\Xi=\lim_{\varepsilon\rightarrow0}\frac{4\varepsilon}{\gamma\left(
\gamma-1\right)  }\left\vert \nabla(\rho^{\varepsilon})^{\frac{\gamma}{2}%
}\right\vert ^{2}.\label{masura}%
\end{equation}
But we also have that%
\begin{align}
\frac{\gamma}{\gamma-1}\operatorname{div}\left(  \rho^{\gamma}u\right)   &
=\operatorname{div}\left(  u\left(  \overline{\rho^{\gamma}}-\rho^{\gamma
}\right)  \right)  -\left(  \overline{\rho^{\gamma}}-\rho^{\gamma}\right)
\operatorname{div}u\nonumber\\
&  -\frac{1}{2}\operatorname{div}\left(  u\rho\left\vert u\right\vert
^{2}\right)  +\mathcal{B}\left(  u,u\right)  -\mathcal{C}\left(  u,u\right)
+gu,\label{diference2}%
\end{align}
such that when taking the difference of $\left(  \text{\ref{diference1}%
}\right)  $ with $\left(  \text{\ref{diference2}}\right)  $ we end up with%
\[
\operatorname{div}\left(  u\left(  \overline{\rho^{\gamma}}-\rho^{\gamma
}\right)  \right)  +\left(  \gamma-1\right)  \left(  \overline{\rho^{\gamma}%
}-\rho^{\gamma}\right)  \operatorname{div}u+\overline{\mathcal{C}\left(
u,u\right)  }-\mathcal{C}\left(  u,u\right)  =-\Xi,
\]
with $\Xi$ the measure defined by $\left(  \text{\ref{masura}}\right)  $. The
proof of the fact that $\nabla u^{\varepsilon}\rightarrow\nabla u$ strongly in
$L^{r}\left(  \mathbb{T}^{3}\right)  $ for all $r\in\lbrack1,\frac{3\left(
\gamma-1\right)  }{\gamma})$ remains essentially the same as in Proposition
\ref{Convergenta_gradient}. Observe that%
\begin{align*}
\operatorname{div}\left(  \rho^{\varepsilon}u^{\varepsilon}\otimes
u^{\varepsilon}\right)  +\varepsilon\nabla u^{\varepsilon}\nabla
\rho^{\varepsilon} &  =\operatorname{div}\left(  \rho^{\varepsilon
}u^{\varepsilon}\right)  u^{\varepsilon}+\rho^{\varepsilon}u^{\varepsilon
}\cdot\nabla u^{\varepsilon}+\varepsilon\nabla u^{\varepsilon}\nabla
\rho^{\varepsilon}\\
&  =\varepsilon\Delta\rho^{\varepsilon}u^{\varepsilon}+\varepsilon\nabla
u^{\varepsilon}\nabla\rho^{\varepsilon}+\rho^{\varepsilon}u^{\varepsilon}%
\cdot\nabla u^{\varepsilon}\\
&  =\varepsilon\operatorname{div}\left(  u^{\varepsilon}\otimes\nabla
\rho^{\varepsilon}\right)  +\rho^{\varepsilon}u^{\varepsilon}\cdot\nabla
u^{\varepsilon}.
\end{align*}
Applying $\operatorname{div}_{\theta}$ in the velocity's equation we obtain
that%
\begin{align*}
\Delta_{\theta}(\mu\operatorname{div}_{\theta}u^{\varepsilon}+\left(
\mu+\lambda\right)  \operatorname{div}u^{\varepsilon}+\xi\ast
\operatorname{div}u^{\varepsilon}-(\rho^{\varepsilon})^{\gamma}) &
=-\operatorname{div}\left(  \nabla\eta\ast\operatorname{div}_{\theta
}u^{\varepsilon}\right)  -\varepsilon\operatorname{div}_{\theta}%
\operatorname{div}\left(  u^{\varepsilon}\otimes\nabla\rho^{\varepsilon
}\right)  \\
&  -\operatorname{div}_{\theta}\left(  \rho^{\varepsilon}u^{\varepsilon}%
\cdot\nabla u^{\varepsilon}\right)  -\operatorname{div}_{\theta}g.
\end{align*}
thus, by denoting%
\[
w^{\varepsilon}\overset{not.}{=}\mu\operatorname{div}_{\theta}u^{\varepsilon
}+\left(  \mu+\lambda\right)  \operatorname{div}u^{\varepsilon}+\xi
\ast\operatorname{div}u^{\varepsilon}-(\rho^{\varepsilon})^{\gamma
}+\varepsilon\Delta_{\theta}^{-1}\operatorname{div}_{\theta}\operatorname{div}%
\left(  u^{\varepsilon}\otimes\nabla\rho^{\varepsilon}-\int_{\mathbb{T}^{3}%
}u^{\varepsilon}\otimes\nabla\rho^{\varepsilon}\right)
\]
using the uniform estimates $\left(  \text{\ref{energy_estimate}}\right)  $ we
get that
\[
w^{\varepsilon}\in W^{1,\frac{3(\gamma-1)}{2\gamma-1}}\left(  \mathbb{T}%
^{3}\right)
\]
such that using the Rellich-Kondrachov theorem we get that%
\[
w^{\varepsilon}\rightarrow w=\mu\operatorname{div}_{\theta}u+\left(
\mu+\lambda\right)  \operatorname{div}u+\xi\ast\operatorname{div}%
u-\overline{\rho^{\gamma}}%
\]
strongly for all $r\in\lbrack1,\frac{3(\gamma-1)}{\gamma})$. Armed with this
piece of information we proceed as in Section \ref{NWS} concerning the
nonlinear weak stability in order to conclude that $\rho^{\gamma}%
=\overline{\rho^{\gamma}}$. This ends the proof of Theorem \ref{Main1}. \appendix

\section{Appendix\label{Section_tools}}

\subsection*{Functional analysis tools}

This section is devoted to a quick recall of the main results from functional
analysis that we used thought the text. Consider $p\in\lbrack1,\infty)$, $g\in
L^{p}\left(  \mathbb{T}^{3}\right)  $ and $\omega\in\mathcal{D}\left(
\mathbb{R}^{3}\right)  $ a smooth, nonnegative, even function compactly
supported in the unit ball centered at the origin and with integral equal to
$1$. For all $\varepsilon>0$, we introduce the averaged functions
\begin{equation}
g_{\varepsilon}=g\ast\omega_{\varepsilon}(x)\qquad\hbox{ where }\qquad
\omega_{\varepsilon}=\frac{1}{\varepsilon^{3}}\omega(\frac{x}{\varepsilon
}).\label{notation_approx}%
\end{equation}
We recall the following classical analysis result
\[
\lim_{\varepsilon\rightarrow0}\left\Vert g_{\varepsilon}-g\right\Vert
_{L^{p}({\mathbb{T}}^{3})}=0.
\]
Moreover, for any multi-index $\alpha$ there exists a constant $C\left(
\varepsilon,\alpha\right)  $ such that
\[
\left\Vert \partial^{\alpha}g_{\varepsilon}\right\Vert _{L^{\infty}}\leq
C\left(  \varepsilon,\alpha\right)  \left\Vert g\right\Vert _{L^{p}}.
\]
Next let us recall the following result concerning the commutator between the
convolution with $\omega_{\varepsilon}$ and the product with a given function.
More precisely, we have that

\begin{proposition}
[Sobolev's inequality]Consider $p\in\lbrack1,3)$ and $g\in W^{1,p}\left(
\mathbb{T}^{3}\right)  $ with $\int_{\mathbb{T}^{3}}g=0$. Then,
\[
\left\Vert g\right\Vert _{L^{p_{\star}}}\leq\left\Vert \nabla g\right\Vert
_{L^{p}}%
\]
where $\frac{1}{p^{\star}}=\frac{1}{p}-\frac{1}{3}$.  
\end{proposition}

\begin{proposition}
\label{Prop_ren1}Consider $\beta\in(1,\infty)$ and $\left(  a,b\right)  $ such
that $a\in L^{\beta}\left(  \mathbb{T}^{3}\right)  $ and $b,\nabla b\in
L^{p}\left(  \mathbb{T}^{3}\right)  $ where $\frac{1}{s}=\frac{1}{\beta}%
+\frac{1}{p}\leq1$. Then, we have%
\[
\lim r_{\varepsilon}\left(  a,b\right)  =0\text{ in }L^{s}\left(
\mathbb{T}^{3}\right)
\]
where
\begin{equation}
r_{\varepsilon}\left(  a,b\right)  =\partial_{i}\left(  a_{\varepsilon
}b\right)  -\partial_{i}\left(  \left(  ab\right)  _{\varepsilon}\right)
,\label{def_reminder}%
\end{equation}
with $i\in\left\{  1,2,3\right\}  $.
\end{proposition}

One also has the following:

\begin{proposition}
\label{Prop_ren2}Consider $2\leq\beta<\infty$ and $\lambda_{0},\lambda_{1}$
such that $\lambda_{0}<1$ and $-1\leq\lambda_{1}\leq\beta/2-1$. Also, consider
$\rho\in L^{\beta}\left(  \mathbb{T}^{3}\right)  $, $\rho\geq0$ a.e. and
$u,\nabla u\in L^{2}\left(  \mathbb{T}^{3}\right)  $ verifying the following
stationary transport equation%
\[
\operatorname{div}\left(  \rho u\right)  =0
\]
in the sense of distributions. Then, for any function $b\in C^{0}\left(
[0,\infty)\right)  \cap C^{1}\left(  \left(  0,\infty\right)  \right)  $ such
that%
\[
\left\{
\begin{array}
[c]{l}%
b^{\prime}\left(  t\right)  \leq ct^{-\lambda_{0}}\text{ for }t\in(0,1],\\
\left\vert b^{\prime}\left(  t\right)  \right\vert \leq ct^{\lambda_{1}}\text{
for }t\geq1
\end{array}
\right.
\]
it holds that%
\begin{equation}
\operatorname{div}\left(  b\left(  \rho\right)  u\right)  +\left\{  \rho
b^{\prime}\left(  \rho\right)  -b\left(  \rho\right)  \right\}
\operatorname{div}u=0. \label{renorm}%
\end{equation}
in the sense of distributions.
\end{proposition}

\noindent The proof of the above results follow by adapting in a
straightforward manner lemmas $6.7.$ and $6.9$ from the book of A.
Novotn\'{y}- I.Stra\v{s}kraba
pages $155-188$. We end up this section with the following theorem that will
be used to prove existence of solutions:

\begin{theorem}
[Schauder-Leray]\label{Schauder_Leray}Let $\mathcal{T}$ be a continuous
compact mapping of a Banach space $\mathcal{B}$ into itself with the property
that there exists a real positive number $M>0$ such that%
\[
\left\Vert x\right\Vert _{\mathcal{B}}\leq M,
\]
for all $x$ such that $x=\lambda\mathcal{T}x$ for some $\lambda\in\left[
0,1\right]  $. Then $\mathcal{T}$ admits a fixed point.
\end{theorem}

\noindent For a proof of this result see Theorem $11.3.$ page $280$ from
\cite{GilTru}.

\subsection{Fourier analysis tools}

In this section, we recall certain results concerning Fourier multiplier
operators on the torus and the whole space and we recall the relation between
them. More precisely, for the rest of the paper of this section we fix a
\textit{bounded} function $m:\mathbb{R}^{n}\backslash\left\{  0\right\}
\rightarrow\mathbb{C}$.

\begin{definition}
We say that $m$ is a $\left(  p,p\right)  $-multiplier on $\mathbb{R}^{n}$ if
the operator $S$ defined by%
\begin{equation}
S\left(  g\right)  =\mathcal{F}^{-1}\left(  m\left(  \xi\right)
\mathcal{F}\left(  g\right)  \right)  , \label{definitie_S}%
\end{equation}
for all tempered distributions $g$ which have the support of their Fourier
transform supported away from $0$ can be extended to an operator that maps
$L^{p}\left(  \mathbb{R}^{n}\right)  $ into itself. The class of all $\left(
p,p\right)  $-multipliers on $\mathbb{R}^{n}$ is denoted $\mathit{M}%
_{p}\left(  \mathbb{R}^{n}\right)  $ and we define the $\mathit{M}_{p}$-norm
of $m$ as being the operatorial norm of the associated operator $S$ i.e.
\[
\left\Vert m\right\Vert _{\mathit{M}_{p}\left(  \mathbb{R}^{n}\right)
}:\overset{def.}{=}\left\Vert S\right\Vert _{\mathcal{L}\left(  L^{p}\left(
\mathbb{R}^{n}\right)  ,L^{p}\left(  \mathbb{R}^{n}\right)  \right)  }.
\]

\end{definition}

\medskip

\noindent In the following we denote $L_{0}^{p}\left(  \mathbb{T}^{n}\right)
$ the closed subspace of $L^{p}\left(  \mathbb{T}^{n}\right)  $ with mean
value $0$.

\begin{definition}
We say that $\left\{  m(k)\right\}  _{k\in\mathbb{Z}^{n}\backslash\left\{
0\right\}  }$ is a $\left(  p,p\right)  $-multiplier on the torus if the
operator $T$ defined by%
\begin{equation}
T\left(  P\right)  \left(  x\right)  =P\left(  x\right)  =\sum_{k\in
\mathbb{Z}^{n}\backslash\left\{  0\right\}  }m\left(  k\right)  a_{k}%
\exp\left(  2\pi ik\cdot x\right)  , \label{definitie_T}%
\end{equation}
for all trigonometric polynomials with zero mean i.e.
\[
P\left(  x\right)  =\sum_{k\in\mathbb{Z}^{n}}a_{k}\exp\left(  2\pi ik\cdot
x\right)  ,
\]
with $\left(  a_{k}\right)  _{k\in\mathbb{Z}^{n}}$ with finite support and
$a_{0}=0$, can be extended to an operator that maps $L_{0}^{p}\left(
\mathbb{T}^{n}\right)  $ into itself. The class of all $\left(  p,p\right)
$-multipliers on the torus is denoted $\mathit{M}_{p}\left(  \mathbb{Z}%
^{n}\right)  $ and we define the $\mathit{M}_{p}$-norm of $m$ as being the
operatorial norm of the associated operator $T$ i.e.
\[
\left\Vert m\right\Vert _{\mathit{M}_{p}\left(  \mathbb{Z}^{n}\right)
}:\overset{def.}{=}\left\Vert S\right\Vert _{\mathcal{L}\left(  L^{p}\left(
\mathbb{T}^{n}\right)  ,L^{p}\left(  \mathbb{T}^{n}\right)  \right)  }.
\]

\end{definition}

\medskip

One of the classical subjects in Fourier analysis tries to capture the
properties that $m$ has to satisfy in order to be a $\left(  p,p\right)
$-Fourier multiplier. In the following, we recall Mihlin's multiplier theorem
that gives a sufficient conditions such that $m$ to be a Fourier multiplier on
$\mathbb{R}^{n}$.

\begin{theorem}
Let $m\left(  \xi\right)  $ be a complex-valued bounded function on
$\mathbb{R}^{n}\backslash\{0\}$ that satisfies Mihlin's condition%
\begin{equation}
\left\vert \partial_{\xi}^{\alpha}m\left(  \xi\right)  \right\vert \leq
A\left\vert \xi\right\vert ^{-\left\vert \alpha\right\vert },
\label{Milhin_condition}%
\end{equation}
for all multi-indices $\left\vert \alpha\right\vert \leq\left[  \frac{n}%
{2}\right]  +1$. Then, for all $p\in\left(  1,\infty\right)  $, $m$ is a
$\left(  p,p\right)  $-multiplier on $\mathbb{R}^{n}$ and there exists a
constant $C_{n}$ depending only on the dimension $n$ such that for all $g\in
L^{p}\left(  \mathbb{R}^{n}\right)  :$
\[
\left\Vert m\right\Vert _{\mathit{M}_{p}\left(  \mathbb{R}^{n}\right)  }\leq
C_{n}\max\left\{  p,\frac{1}{p-1}\right\}  \left(  A+\left\Vert m\right\Vert
_{L^{\infty}\left(  \mathbb{R}^{n}\right)  }\right)  \left\Vert g\right\Vert
_{L^{p}\left(  \mathbb{R}^{n}\right)  }\text{.}%
\]

\end{theorem}

\noindent A proof of this result can be found in L. Grafakos's book, see
\cite{Grafakos2008} Theorem $5.2.7.$, page $367$.

\begin{remark}
\label{norma}One can check by direct calculation that $m:\mathbb{R}%
^{3}\backslash\left\{  0\right\}  \rightarrow\mathbb{R}$ defined by%
\[
m\left(  \xi\right)  =\frac{\left\vert \xi_{3}\right\vert ^{2}}{a_{1}%
\left\vert \xi_{1}\right\vert ^{2}+a_{2}\left\vert \xi_{2}\right\vert
^{2}+a_{3}\left\vert \xi_{3}\right\vert ^{2}},
\]
verifies the Milhin condition $\left(  \text{\ref{Milhin_condition}}\right)  $
with $A=\max\left\{  A_{0},A_{1},A_{2}\right\}  $ with
\[
A_{0}=\frac{1}{a_{3}},
\]
and
\[
A_{1}=\max\left\{  \frac{\sqrt{a_{1}}}{a_{3}},\frac{\sqrt{a_{2}}}{a_{3}}%
,\frac{1}{\sqrt{a_{3}}}\right\}  \frac{1}{\sqrt{\min\left\{  a_{1},a_{2}%
,a_{3}\right\}  }},
\]
and
\[
A_{2}=\max\left\{  \frac{a_{1}}{a_{3}},\frac{a_{2}}{a_{3}},1\right\}  \frac
{1}{\min\left\{  a_{1},a_{2},a_{3}\right\}  },
\]
where each $A_{i}$ represents the constant appearing in the $\left(
\text{\ref{Milhin_condition}}\right)  $ respectively for the $|\alpha|=0$,
$|\alpha|=1$ and $|\alpha|=2$ derivatives. Milhin's theorem implies that $m$
is a Fourier multiplier on $\mathbb{R}^{n}$.
\end{remark}

\begin{definition}
Let $\xi_{0}\in\mathbb{R}^{n}$. A bounded function $m$ on $\mathbb{R}^{n}$ is
called regulated at the point $\xi_{0}$ if%
\[
\lim_{\varepsilon\rightarrow0}\int_{\left\vert t\right\vert \leq\varepsilon
}\left(  m\left(  \xi_{0}-\xi\right)  -m\left(  \xi_{0}\right)  \right)
d\xi=0.
\]

\end{definition}

Obviously, if $m$ is continuous in $\xi_{0}$ then $m$ is regulated at the
point $\xi_{0}$. The following result is the key point in transferring the
Milhin theorem on the torus:

\begin{lemma}
\label{LemmaMultiplier}Let $T$ be a operator on $\mathbb{R}^{n}$ whose
multiplier is $m\left(  \xi\right)  $ and let $S$ be the operator on
$\mathbb{T}^{n}$ whose multiplier is the sequence $\left\{  m\left(  k\right)
\right\}  _{k\in\mathbb{Z}^{n}}$. Assume that $m\left(  \xi\right)  $ is
regular at every point in $\mathbb{Z}^{n}\backslash\left\{  0\right\}  $.
Suppose that $P$ and $Q$ are trigonometric polynomials on $\mathbb{T}^{n}$ and
let $L_{\varepsilon}\left(  x\right)  =\exp\left(  -\pi\varepsilon\left\vert
x\right\vert ^{2}\right)  $ for $x\in\mathbb{R}^{n}$ and $\varepsilon>0$. Then
the following identity is valid whenever $\alpha,\beta>0$ and $\alpha
+\beta=1:$%
\[
\lim_{\varepsilon\rightarrow0}\varepsilon^{\frac{n}{2}}\int_{\mathbb{R}^{n}%
}T\left(  PL_{\varepsilon\alpha}\right)  \left(  x\right)  \overline
{(QL_{\varepsilon\beta})\left(  x\right)  }dx=\int_{\mathbb{T}^{n}}S\left(
P\right)  \left(  x\right)  \overline{Q\left(  x\right)  }dx.
\]

\end{lemma}

The above lemma is different from Lemma $3.6.8.$ from \cite{Grafakos2008} page
$224$ only in one aspect: as we are looking to obtain results for functions
with mean value $0$, we may ask $m$ to be regulated at every point of
$\mathbb{Z}^{n}\backslash\left\{  0\right\}  $ instead of $\mathbb{Z}^{n}$.
However, the proof is the same word for word. Finally, we are able to ass the following

\begin{theorem}
\label{multipliers}Suppose that $m:\mathbb{R}^{n}\backslash\left\{  0\right\}
\rightarrow\mathbb{C}$ $\left(  p,p\right)  $-Fourier multiplier on
$\mathbb{R}^{n}$ for some $p\in\lbrack1,\infty)$ and that it is regulated at
every point in $\mathbb{Z}^{n}\backslash\left\{  0\right\}  $. Then, $\left\{
m\left(  k\right)  \right\}  _{k\in\mathbb{Z}^{n}\backslash\left\{  0\right\}
}$ defines a $\left(  p,p\right)  $-Fourier multiplier and%
\[
\left\Vert \left\{  m\left(  k\right)  \right\}  _{k\in\mathbb{Z}%
^{n}\backslash\left\{  0\right\}  }\right\Vert _{\mathit{M}_{p}\left(
\mathbb{Z}^{n}\right)  }\leq\left\Vert m\right\Vert _{\mathit{M}_{p}\left(
\mathbb{R}^{n}\right)  }.
\]

\end{theorem}

Theorem \ref{multipliers} is a restatement of Theorem $3.6.7.$ from
\cite{Grafakos2008} page $224$ in the context of $L^{p}$ functions with mean
value $0$. The proof is a consequence of the fact that the $L^{p}$-norm of a
function can be expressed by duality as the supremum over all trigonometric
functions with $L^{p^{\prime}}$ norm less than $1$ combined with
\ref{LemmaMultiplier}. The interested reader is referred to
\cite{Grafakos2008} pages $224-225$ for a complete proof.

\bigskip

\noindent We use Theorem \ref{multipliers} and Remark \ref{norma} in order to
estimate the norm of the Fourier multiplier operator on the torus%
\[
\left(  Id-\left(  2\mu+\lambda\right)  \left(  \mu\Delta_{\theta}+\left(
\mu+\lambda\right)  \Delta\right)  ^{-1}\Delta\right)
\]
whose multiplier is
\[
m\left(  \xi_{1},\xi_{2},\xi_{3}\right)  =\frac{\theta\mu\left\vert \xi
_{3}\right\vert ^{2}}{\left(  2\mu+\lambda\right)  \left(  \left\vert \xi
_{1}\right\vert ^{2}+\left\vert \xi_{2}\right\vert ^{2}\right)  +\left(
\left(  2+\theta\right)  \mu+\lambda\right)  \left\vert \xi_{3}\right\vert
^{2}}.
\]
According to Remark \ref{norma} and Theorem \ref{multipliers}, taking in
consideration that $\theta>-1$ and after some long but straightforward
computations we obtain that there exists a numerical constant $C>0$ such that
\begin{equation}
\left\Vert m\right\Vert _{\mathit{M}_{p}\left(  \mathbb{Z}^{n}\right)  }\leq
C(1+|\theta|)\left\vert \theta\right\vert \mu\frac{\left(  2\lambda
+\mu\right)  }{(\lambda+\mu)^{2}}.\label{norma_multiplicator_fourier}%
\end{equation}

\medskip

\noindent\textbf{Acknowledgments.} D. Bresch and C. Burtea are supported by
the SingFlows project, grant ANR-18-CE40-0027 and D. Bresch is also supported
by the Fraise project, grant ANR-16-CE06- 0011 of the French National Research
Agency (ANR).

\end{document}